\documentclass[12pt]{amsart}
\usepackage{enumerate}
\usepackage{amscd}

\makeatletter
\let\@@enum@org\@@enum@
\def\@@enum@[#1]{\@@enum@org[\normalfont #1]}
\makeatother
\usepackage{hyperref}
\hypersetup{
    pdftitle=   {Rank-width  and Well-quasi-ordering of Skew-Symmetric or
  Symmetric Matrices},
    pdfauthor=  {Sang-il Oum}
}
\newtheorem{THM}{Theorem}[section]
\newtheorem*{THMMAIN}{Theorem \ref{thm:main}}
\newtheorem*{THMMAINMAT}{Theorem \ref{thm:mainmatrix}}
\newtheorem{LEM}[THM]{Lemma}
\newtheorem{COR}[THM]{Corollary}
\newtheorem{PROP}[THM]{Proposition}
\theoremstyle{remark}

\newtheorem*{REM}{Remark}
\newcommand\form[1]{\left\langle #1\right\rangle}
\newcommand\rk{\operatorname{rank}}

\newcommand\MM{\mathcal M}
\newcommand\FF{\mathcal F}
\newcommand\bw{\operatorname{bw}}
\newcommand\dup[1]{{#1}^*}
\newcommand\ddown[1]{{#1}_*}
\newcommand\rwd{\operatorname{rwd}}
\usepackage{stmaryrd}
\newcommand\delete{\bbslash}
\newcommand\contract{\sslash} 
\usepackage{bbm}
\newcommand\F{\mathbbm{F}}

\begin{document}
\title[Rank-width and Well-quasi-ordering]
{Rank-width  and Well-quasi-ordering of Skew-Symmetric or
  Symmetric Matrices}
\author{Sang-il Oum}
\address{Department of Mathematical Sciences, KAIST, 335 Gwahangno, Yuseong-gu, Daejeon 305-701, Republic of Korea}
\email{sangil@kaist.edu}
\thanks{Supported by Basic Science Research Program through the National Research Foundation of Korea (NRF) funded by the Ministry of Education, Science and Technology (2010-0001655) and TJ Park Junior Faculty Fellowship.}
\date{July 22, 2010}
\begin{abstract}
We prove that every infinite sequence
  of skew-symmetric or symmetric matrices  $M_1$, $M_2$, $\ldots$  
 over a fixed finite field 
  must have a pair $M_i$, $M_j$ $(i<j)$ such
  that $M_i$ is isomorphic to a principal submatrix of 
  the Schur complement of a nonsingular principal submatrix in $M_j$,
  if those matrices have bounded rank-width.
  This generalizes three theorems on
  well-quasi-ordering of graphs or matroids admitting good tree-like decompositions;
  (1) Robertson and Seymour's theorem for
  graphs of bounded tree-width, 
  (2) Geelen, Gerards, and Whittle's theorem
  for matroids representable over a fixed finite field having bounded
  branch-width, 
  and (3) Oum's theorem for graphs of bounded rank-width with respect
  to pivot-minors.
\end{abstract}
\maketitle
\section{Introduction}

For a $V_1\times V_1$ matrix~$A_1$
and a $V_2\times V_2$ matrix~$A_2$, 
an \emph{isomorphism} $f$ from $A_1$ to $A_2$ is a bijective function that
maps $V_1$ to $V_2$ such that 
the $(i,j)$ entry of $A_1$ is equal to the $(f(i),f(j))$ entry of $A_2$
for all $i,j\in V_1$.
Two square matrices $A_1$, $A_2$ are \emph{isomorphic} if there is an
isomorphism from $A_1$ to $A_2$. Note that an isomorphism allows
permuting rows and columns simultaneously.
For a $V\times V$ matrix $A$ and a subset $X$ of its ground set $V$, 
we write $A[X]$ to denote the principal submatrix of $A$ induced by
$X$.
Similarly, we write $A[X,Y]$ to denote the $X\times Y$ submatrix of $A$.
Suppose that a $V\times V$ matrix $M$ has the following form:
\[
M=
\bordermatrix{
&\quad Y\quad & V\setminus Y\cr
Y & A& B\cr
V\setminus Y& C & D
}.
\]
If $A=M[Y]$ is nonsingular, then we define  the \emph{Schur complement
$(M/A)$ of $A$ in $M$}  to be 
\[(M/A)= D-CA^{-1}B.\]
(If $Y=\emptyset$, then $A$ is nonsingular and $(M/A)=M$.)
Notice that if $M$ is skew-symmetric or symmetric, then  $(M/A)$ 
is skew-symmetric or symmetric, respectively.

We prove that skew-symmetric or symmetric matrices over a fixed finite
field
are \emph{well-quasi-ordered} under the relation 
defined in terms of taking a principal submatrix and a Schur
complement, 
if they have bounded \emph{rank-width}. 
Rank-width of a skew-symmetric or symmetric matrix 
will be defined precisely in Section~\ref{sec:prelim}.
Roughly speaking,
it is 
a measure to describe how easy it is to decompose the matrix into a tree-like structure
so that the connecting matrices have small rank. Rank-width of
matrices generalizes rank-width of simple graphs introduced by Oum and
Seymour~\cite{OS2004},
and branch-width of graphs and matroids by Robertson and Seymour~\cite{RS1991}.
Here is our main theorem.
\begin{THMMAINMAT}
  Let $\F$ be a finite field and let $k$ be a constant. 
  Every infinite sequence $M_1$, $M_2$, $\ldots$  
  of skew-symmetric or
  symmetric matrices over $\F$
  of rank-width at most $k$
 has a pair $i<j$ such that 
  $M_i$ is isomorphic to a principal submatrix of 
 $(M_j/A)$
  for some nonsingular principal submatrix $A$ of $M_j$.
\end{THMMAINMAT}

It may look like a  purely linear algebraic result.
However, it implies
the following
well-quasi-ordering theorems on graphs
and matroids admitting  `good tree-like decompositions.'
\begin{itemize}
\item (Robertson and Seymour~\cite{RS1991}) Every infinite sequence
  $G_1$, $G_2$, $\ldots$ of graphs of bounded tree-width has a pair
  $i<j$ such that $G_i$ is isomorphic to a minor of $G_j$. 
\item (Geelen, Gerards, and Whittle~\cite{GGW2002}) Every
  infinite sequence $M_1$, $M_2$, $\ldots$ of matroids representable
  over a fixed finite field having bounded branch-width
  has a pair $i<j$ such that $M_i$ is isomorphic to a minor of $M_j$.
\item (Oum~\cite{Oum2004a}) Every infinite sequence
  $G_1$, $G_2$, $\ldots$ of simple graphs of bounded rank-width has a pair
  $i<j$ such that $G_i$ is isomorphic to a pivot-minnor of $G_j$. 
\end{itemize}

We ask, as an open problem, whether the requirement on  rank-width is necessary in 
Theorem~\ref{thm:mainmatrix}.
It is likely that our theorem for matrices of bounded
rank-width is a step towards this problem, as Roberson and Seymour
also started with graphs of bounded tree-width. If we have a positive
answer, then this would imply Robertson and Seymour's graph
minor theorem \cite{RS2004a} 
as well as an open problem on the well-quasi-ordering of matroids
representable over a fixed finite field~\cite{GGW2006a}.

A big portion of this paper is devoted to introduce Lagrangian
chain-groups
and prove their relations to skew-symmetric or symmetric
matrices.
One can regard Sections~\ref{sec:lagrangian} and
\ref{sec:matrix} as an almost separate paper introducing
Lagrangian chain-groups, 
their matrix representations, 
and their relations to delta-matroids.
In particular, Lagrangian chain-groups provide an alternative definition of
representable delta-matroids. 
The situation is comparable to Tutte chain-groups,\footnote{We call
  Tutte's chain-groups 
  as \emph{Tutte chain-groups} to 
  distinguish from chain-groups defined in Section~\ref{sec:lagrangian}.}
introduced by
Tutte~\cite{Tutte1956}. Tutte~\cite{Tutte1958} showed that a matroid
is representable over a
field $\F$
 if and only if it is representable by a Tutte chain-group over $\F$.
We prove an analogue of his theorem; \emph{a delta-matroid is representable
over a field $\F$ if and only if it is representable by a Lagrangian
chain-group over $\F$.}
We believe that the notion of  Lagrangian chain-groups 
will be useful to extend the matroid theory to representable delta-matroids.

To prove well-quasi-ordering, we work on Lagrangian chain-groups
instead of skew-symmetric or symmetric matrices
for the convenience.
The main proof of the well-quasi-ordering of Lagrangian chain-groups
is in Sections~\ref{sec:tutte} and
\ref{sec:wqochain}. Section~\ref{sec:tutte} proves a theorem
generalizing Tutte's linking theorem for matroids, which in turn
generalizes Menger's theorem. 
The proof idea in Section~\ref{sec:wqochain} is similar to the
proof of Geelen, Gerards, and Whittle's theorem~\cite{GGW2002} for
representable matroids.


The last two sections discuss how the result on Lagrangian
chain-groups imply our main theorem and its other corollaries. 
Section~\ref{sec:wqomat} formulates the result of Section~\ref{sec:wqochain}
in terms of skew-symmetric or symmetric matrices with respect to the
Schur complement and explain its implications for representable
delta-matroids and simple graphs of bounded rank-width.
Section~\ref{sec:matroid} explains why our theorem implies the
theorem for representable matroids by Geelen, Gerards, and Whittle~\cite{GGW2002}
via Tutte chain-groups.

\section{Preliminaries}\label{sec:prelim}
\subsection{Matrices}
For two sets $X$ and $Y$, we write $X\Delta Y=(X\setminus Y)\cup
(Y\setminus X)$. A $V\times V$ matrix $A$ is called \emph{symmetric} if $A=A^t$, 
\emph{skew-symmetric} if $A=-A^t$
and all of its diagonal entries are zero. 
We require each diagonal entry of a skew-symmetric matrix to be zero,
even if the underlying field has characteristic $2$.

Suppose that a $V\times V$ matrix $M$ has the following form:
\[
M=
\bordermatrix{
& \quad Y\quad & V\setminus Y\cr
Y & A& B\cr
V\setminus Y& C & D
}.
\]
If $A=M[Y]$ is nonsingular, then we define a matrix $M*Y$ by 
\[
M*Y=
\bordermatrix{
&\quad Y\quad & V\setminus Y \cr
Y&A^{-1}  & A^{-1}B \cr
V\setminus Y&-CA^{-1} & (M/A)
}.
\]
This operation is called a \emph{pivot}. In the literature, it has
been called a \emph{principal pivoting}, a \emph{principal pivot
  transformation}, and other various names; we refer to the survey by
Tsatsomeros~\cite{Tsatsomeros2000}.

Notice that if $M$ is skew-symmetric, then so is $M*Y$.
If $M$ is symmetric, then so is $(I_Y)(M*Y)$,
where $I_Y$ is a diagonal matrix such that the diagonal entry indexed by an
element in $Y$ is $-1$ and all other diagonal entries are $1$.

The following theorem implies that 
$(M*Y)[X]$ is nonsingular if and only if 
$M[X\Delta Y]$ is nonsingular.  
\begin{THM}[Tucker \cite{Tucker1960}]\label{thm:tucker}
  Let $M[Y]$ be a nonsingular principal submatrix of 
  a $V\times V$ matrix $M$. Then for all
  $X\subseteq V$, 
  \[
  \det(M*Y)[X]=\det M[Y\Delta X] / \det M[Y].
  \]
\end{THM}
\begin{proof}
  See Bouchet's proof in Geelen's thesis paper \cite[Theorem 2.7]{Geelen1995}.
\end{proof}

\subsection{Rank-width}
A tree is called \emph{subcubic} if every vertex has at most three
incident edges.
We define \emph{rank-width} of a skew-symmetric or symmetric $V\times
V$ matrix $A$ over a field $\F$ by 
rank-decompositions  as follows.
A \emph{rank-decomposition} of $A$ is a
pair $(T,\mathcal L)$ of a subcubic tree $T$ and a bijection
$\mathcal L:V\rightarrow \{t:\text{$t$ is a leaf of $T$}\}$.
For each edge $e=uv$ of the tree $T$, the connected components of $T\setminus e$ 
form a partition $(X_e,Y_e)$ of the leaves of $T$ 
and we call $\rk A[\mathcal L^{-1}(X_e),\mathcal L^{-1}(Y_e)]$ 
the \emph{width} of $e$.
The \emph{width} of a rank-decomposition $(T,\mathcal L)$ is the maximum
width of all edges of $T$.
The \emph{rank-width} $\rwd(A)$ of a skew-symmetric or symmetric
$V\times V$ matrix $A$ over $\F$
is the minimum
width of all its rank-decompositions.
(If $|V|\le 1$, then we define that $\rwd(A)=0$.)

\subsection{Delta-matroids}
Delta-matroids were introduced by Bouchet~\cite{Bouchet1987d}.
A \emph{delta-matroid} is a pair $(V,\FF)$ of a finite set $V$
and a \emph{nonempty} collection $\FF$ of subsets of $V$ such that the following
\emph{symmetric exchange axiom} holds.
\begin{multline*}
\text{
  If $F,F'\in \FF$ and $x\in F\Delta F'$, }\\
\text{then 
  there exists $y\in F\Delta F'$ such that $F\Delta\{x,y\}\in \FF$.
}
\tag{SEA}
\end{multline*}
A member of $\FF$ is called \emph{feasible}.
A delta-matroid is \emph{even}, if cardinalities of 
all feasible sets have the same parity.

Let $\MM=(V,\FF)$ be a delta-matroid.
For a subset $X$ of $V$,  
it is easy to see that $\MM\Delta X=(V,\FF\Delta X)$ is also a delta-matroid, where
$\FF\Delta X=\{F\Delta X:F\in \FF\}$;
this operation is referred to  as \emph{twisting}. 
Also, $\MM\setminus X=(V\setminus X,\FF\setminus X)$ defined by
$\FF\setminus X=\{F\subseteq V\setminus X:F\in \FF\}$ is a
delta-matroid if $\FF\setminus X$ is  nonempty;
we refer to this operation as \emph{deletion}.
Two delta-matroids $\MM_1=(V,\FF_1)$, $\MM_2=(V,\FF_2)$ are 
called \emph{equivalent} if there exists $X\subseteq V$ such that 
$\MM_1=\MM_2\Delta X$.
A delta-matroid that comes from $\MM$ by twisting and/or deletion is
called a \emph{minor} of $\MM$.

\subsection{Representable delta-matroids}
For  a $V\times V$ skew-symmetric or symmetric matrix $A$ over a field $\F$,
let 
\[\FF(A)=\{X\subseteq V:  \text{$A[X]$ is nonsingular}\}\] and $\MM(A)=(V,\FF(A))$. Bouchet
\cite{Bouchet1988b} showed that $\MM(A)$ forms a delta-matroid. 
We call a delta-matroid \emph{representable} over a field $\F$
or \emph{$\F$-representable}
if it  is equivalent to $\MM(A)$
for some  skew-symmetric or symmetric matrix $A$ over $\F$. We also say
that $\MM$ is represented by $A$ if $\MM$ is equivalent to $\MM(A)$.

Twisting (by feasible sets) and deletions are both natural operations
for representable delta-matroids. For $X\subseteq V$, 
$\MM(A)\setminus X=\MM(A[V\setminus X])$, and for a feasible set  $X$,
$\MM(A)\Delta X=\MM(A*X)$ 
by Theorem~\ref{thm:tucker}.
Therefore minors of a $\F$-representable delta-matroid
are $\F$-representable \cite{BD1991}.

\subsection{Well-quasi-order.}
In general, we say that a binary relation $\le$ on a set  $X$ 
is a \emph{quasi-order} if
it is reflexive and transitive.
For a quasi-order $\le$, we say  
``$\le$ is a \emph{well-quasi-ordering}''
or ``$X$ is \emph{well-quasi-ordered} by $\le$''
if for every infinite sequence $a_1,a_2,\ldots$ of elements of $X$,
there exist $i<j$ such that $a_i\le a_j$.
For more detail, see Diestel \cite[Chapter 12]{Diestel2005}.

\section{Lagrangian chain-groups}\label{sec:lagrangian}
\subsection{Definitions}
\label{subsec:def}
If $W$ is a vector space with a bilinear form $\form{\,,\,}$ 
and $W'$ is a subspace of $W$ satisfying 
\[ \form{x,y}=0 \text{ for all }x,y\in W',\]
then $W'$ is called \emph{totally isotropic}.
A vector $v\in W$ is called \emph{isotropic} if $\form{v,v}=0$.
A well-known theorem in linear algebra states that if a bilinear
form $\form{\,,\,}$ is non-degenerate in $W$ 
and $W'$ is a totally isotropic
subspace of $W$, then $\dim(W)= \dim(W')+\dim(W'^\bot)\ge2\dim(W')$
because $W'\subseteq W'^\bot$.

Let $V$ be a finite set and  $\F$ be a field.
Let $K=\F^2$ be a two-dimensional vector space over $\F$.
Let $b^+\left(\tbinom{a}{b},\tbinom{c}{d}\right)=ad+bc$ and
$b^-(\tbinom{a}{b},
\tbinom{c}{d})=ad-bc$ be bilinear
forms on $K$.
We assume that $K$ is equipped with a bilinear form $\form{\,,\,}_K$
that is either $b^+$ or $b^-$. Clearly $b^+$ is symmetric and $b^-$
is skew-symmetric.

A \emph{chain} on $V$ to $K$ is a mapping $f:V\rightarrow K$.
If $x\in V$, the element $f(x)$ of $K$ is called the \emph{coefficient}
of $x$ in $f$. If $V$ is nonnull, there is a \emph{zero chain} on $V$
whose coefficients are $0$.
When $V$ is null, we say that there is just one chain on $V$ to $K$
and we call it a zero chain. 

The \emph{sum} $f+g$ of two chains $f$, $g$ is the chain on $V$ satisfying 
$(f+g)(x)=f(x)+g(x)$  for all $x\in V$.
If $f$ is a chain on $V$ to $K$ and $\lambda \in \F$, the \emph{product}
$\lambda f$ is a chain on $V$ such that 
$(\lambda f)(x)=\lambda f(x) $ for all $x\in V$.
It is easy to see that the set of all chains on $V$ to $K$, denoted by
$K^V$, is a vector space. 
We give a bilinear form $\form{\,,\,}$ to $K^V$ as following:
\[
\form{f,g}=\sum_{x\in V} \form{f(x),g(x)}_K.
\]
If $\form{f,g}=0$, we say that the chains $f$ and $g$ are
\emph{orthogonal}.
For a subspace $L$ of $K^V$, we write $L^\bot$ for the set of all
chains orthogonal to every chain in $L$.

A \emph{chain-group} on $V$ to $K$ is a subspace of $K^V$.
A chain-group is called \emph{isotropic} if 
it is a totally isotropic subspace.
It is called \emph{Lagrangian} if 
it is isotropic and 
has dimension $|V|$.
We say a chain-group $N$ is over a field $\F$ if $K$ is obtained from
$\F$ as described above.

A \emph{simple isomorphism} from a chain-group  $N$ on $V$ to $K$ to another
chain-group $N'$ on $V'$ to $K$ is defined as a bijective function
$\mu:V\rightarrow V'$ satisfying that 
$N=\{f\circ \mu : f\in N'\}$
where $f\circ\mu$ is a chain on $V$ to $K$ such that
$(f\circ \mu)(x)=f(\mu(x))$ for all $x\in V$.
We require both $N$ and $N'$ have the same type of bilinear forms on
$K$, that is either skew-symmetric or symmetric.
A chain-group $N$ on $V$ to $K$ 
is \emph{simply isomorphic} to another chain-group
$N'$ on $V'$ to $K$ 
if there is a simple isomorphism from $N$ to $N'$. 
\begin{REM}
  Bouchet's definition \cite{Bouchet1988b} of isotropic chain-groups
  is slightly more general than ours, since he allows
  $\form{\tbinom{a}{b},\tbinom{c}{d}}_K=-ad\pm bc$. His
  notation, however, is different; he uses $\F^{V'}$ instead of $K^V$
  where 
  $V'$ is a union of $V$ and its disjoint copy $V^\sim$. Since $K=\F^2$,
  two definitions are equivalent. Our notation has advantages  which we
  will see in the next subsection. Bouchet's notation
  also  has its own virtues because, in Bouchet's sense, 
  isotropic chain-groups are Tutte chain-groups.
  Strictly speaking, our 
  isotropic chain-groups are not Tutte chain-groups, because 
  we define chains differently.
  We are mainly interested in Lagrangian chain-groups because they are closely
  related to representable delta-matroids. We note that the notion of
  Lagrangian chain-groups is motivated by Tutte's chain-groups and
  Bouchet's isotropic systems~\cite{Bouchet1987a}.
\end{REM}

\subsection{Minors}
Consider a subset $T$ of $V$. If $f$ is a chain on $V$ to $K$, 
we define its \emph{restriction} $f\cdot T$ to $T$ as 
the chain on $T$ such that 
$(f\cdot T)(x)=f(x)$ for all $x\in T$.
For a chain-group $N$ on $V$, 
\[N\cdot T=\{f\cdot T: f\in N\}\]
is a chain-group on $T$ to $K$. We note that
$N\cdot T$  is not necessarily isotropic, even if $N$ is isotropic.
We write \[N\times T=\{f\cdot T: f\in N,~f(x)=0 \text{ for all }x\in
V\setminus T\}.\]

For a chain-group $N$ on $V$, we define 
\[N\delete T=\{ f\cdot (V\setminus T): f\in N, \form{f(x),\tbinom 1
  0}_K=0 \text{ for all } x\in T\}.\]
We call this the \emph{deletion}.
Similarly we define 
\[N\contract T=\{ f\cdot (V\setminus T): f\in N, \form{f(x),\tbinom 0
  1}_K=0 \text{ for all } x\in T\}.\]
We call this the \emph{contraction}.
We refer to a chain-group of the form 
$N\contract X\delete Y$ on $V\setminus (X\cup Y)$ 
as a \emph{minor} of $N$.

\begin{PROP}
  A minor of a minor of a  chain-group $N$ on $V$ to $K$ is a minor of $N$.
\end{PROP}
\begin{proof}
  We can deduce this from the following easy facts.
  \begin{align*}
    N\contract X\contract Y &= N\contract (X\cup Y),\\
    N\contract X\delete Y &=N\delete Y\contract X,\\
    N\delete X\delete Y&=N\delete (X\cup Y). \qedhere
  \end{align*}
\end{proof}
\begin{LEM}\label{lem:basic}
  Let $x,y\in K$.
  If $x\in K$ is isotropic, $x\neq 0$,  and $\form{x,y}_K=0$, then 
  $y=cx$ for some $c\in \F$.
\end{LEM}
\begin{proof}
  Since $\form{\,,\,}_K$ is nondegenerate, there exists a vector
  $x'\in K$ such that $\form{x,x'}_K\neq0$.
  Hence $\{x,x'\}$ is a basis of $K$.
  Let $y=cx+dx'$ for some $c,d\in \F$. 
  Since $\form{x,cx+dx'}_K=d\form{x,x'}_K=0$, we deduce $d=0$.
\end{proof}
\begin{PROP}\label{prop:minorisotropic}
  A minor of an isotropic chain-group on $V$ to $K$ is isotropic.
\end{PROP}
\begin{proof}
 By Lemma~\ref{lem:basic}, if
  $\form{x,\tbinom10}_K=\form{y,\tbinom10}_K=0$, then $\form{x,y}_K=0$
  and similarly if
  $\form{x,\tbinom01}_K=\form{y,\tbinom01}_K=0$, then
  $\form{x,y}_K=0$.
  This easily implies the lemma.
\end{proof}

We will  prove that every minor of a
Lagrangian chain-group is Lagrangian in the next section.

\subsection{Algebraic duality}
For an element $v$ of a finite set $V$,
if $N$ is a chain-group on $V$ to $K$ and $B$ is a basis of $N$, then
we may assume that the coefficient at $v$ of every chain in $B$ 
is zero except at most two chains in $B$ because $\dim(K)=2$.
So, it is clear that dimensions of  $N\times (V\setminus\{v\})$,
$N\cdot (V\setminus\{v\})$, $N\delete \{v\}$, and $N\contract \{v\}$ 
are at least $\dim(N)-2$. In this subsection, we discuss conditions
for those chain-groups to have dimension $\dim(N)-2$, $\dim(N)-1$, or
$\dim(N)$. Note that we do not assume that $N$ is isotropic.
\begin{THM}\label{thm:algdual}
  If $N$ is  a chain-group on $V$ to $K$ and $X\subseteq V$,
  then 
  \[ (N\cdot X)^\bot = N^\bot \times X . \]
\end{THM}
\begin{proof}
  (Tutte {\cite[Theorem VIII.7.]{Tutte1984}})
  Let $f\in (N\cdot X)^\bot$.
  There exists a chain $f_1$  on $V$ to $K$ such that $f_1\cdot X=f$ and
    $f_1(v)=0$ for all $v\in V\setminus X$.
    Since $\form{f_1,g}=\form{f,g\cdot X}=0$ for all $g\in N$,
    we have $f\in N^\bot \times X$.

    Conversely, if $f\in N^\bot \times X$, it is the restriction
    to $X$ of a chain $f_1$ of $N^\bot$ specified as above. Hence
    $\form{f,g\cdot X}=\form{f_1,g}=0$ for all $g\in N$. Therefore $f\in (N\cdot X)^\bot$.
\end{proof}
\begin{LEM}\label{lem:linalg}
  Let $N$ be a chain-group on $V$ to $K$.
  If $X\cup Y=V$ and $X\cap Y=\emptyset$, then 
  \[
  \dim(N\cdot X)+\dim(N\times Y)=\dim(N).
  \]
\end{LEM}
\begin{proof}
  Let $\varphi:N\rightarrow N\cdot X$ be a linear transformation 
  defined by $\varphi(f)=f\cdot X$.
  The kernel $\ker(\varphi)$ of this transformation
  is the set of all chains $f$ in $N$ having $f\cdot X=0$.
  Thus, $\dim(\ker(\varphi))=\dim(N\times Y)$. Since $\varphi$ is
  surjective, we deduce that 
  $\dim(N\cdot X)=\dim(N)-\dim(N\times Y)$.
\end{proof}
For $v\in V$,
let $\dup v$, $\ddown v$ be chains on $V$ to $K$ such that
\begin{align*}
\dup v(v)&=\tbinom{1}{0},\quad
\ddown v(v)=\tbinom{0}{1},\\
\dup v(w)&=\ddown v(w)=0 \quad\text{for all }w\in V\setminus \{v\}.
\end{align*}

\begin{PROP}\label{prop:minordim}
  Let $N$ be a chain-group on $V$ to $K$ and $v\in V$. Then
  \begin{align*}
  \dim(N\delete \{v\})&=
  \begin{cases}
    \dim N & \text{if } \dup v\notin N,\dup v\in N^\bot,\\
    \dim N-2 &\text{if }\dup v\in N,\dup v\notin N^\bot,\\
    \dim N-1 &\text{otherwise,}
  \end{cases}\\
 \dim(N\contract \{v\})&=
  \begin{cases}
    \dim N & \text{if } \ddown v\notin N,\ddown v\in N^\bot,\\
    \dim N-2 &\text{if }\ddown v\in N,\ddown v\notin N^\bot,\\
    \dim N-1 &\text{otherwise.}
  \end{cases}
  \end{align*}
\end{PROP}
\begin{proof}
  By symmetry, it is enough to show for $\dim(N\delete \{v\})$.
  Let $N'=\{f\in N: \form{f(v),\tbinom 1 0}_K=0\}$.
  By definition, $N\delete \{v\}=N'\cdot (V\setminus \{v\})$.

  Observe that $N'=N$ if and only if $\dup v\in N^\bot$.
  If $N'\neq N$, then there is a chain $g$ in $N$ 
  such that $\form{g(v),\tbinom10}_K\neq 0$. 
  Then, for every chain $f\in N$, there exists $c\in \F$ such
  that 
  $f-cg\in N'$. Therefore $\dim(N')=\dim N-1$ if $\dup v\notin N^\bot$
  and $\dim (N')=\dim N$ if $\dup v\in N^\bot$.

  By Lemma~\ref{lem:linalg}, 
  $\dim(N'\cdot (V\setminus\{v\}))=\dim N'-\dim (N'\times\{v\})$. 
  Clearly, $\dim (N'\times\{v\})=0$ if $\dup v \notin N$
  and $\dim (N'\times\{v\})=1$ if $\dup v\in N$.
  This concludes the proof.
  %
\end{proof}
\begin{COR}\label{cor:minordim}
  If $N$ is an isotropic chain-group on $V$ to $K$
  and $M$ is a minor of $N$ on $V'$, 
  then 
  \[|V'|-\dim M \le |V|-\dim N.\]
\end{COR}
\begin{proof}
  We proceed by induction on $|V\setminus V'|$.
  Since $N$ is isotropic, every minor of $N$ is isotropic by
  Proposition~\ref{prop:minorisotropic}.
 Since $\dup v\notin N\setminus N^\bot$ and $\ddown v\notin N\setminus
  N^\bot$,
  $\dim(N)-\dim(N\delete \{v\})\in \{0,1\}$ and
  $\dim(N)-\dim(N\contract \{v\})\in \{0,1\}$. 
  So $|V\setminus\{v\}|-\dim (N\delete\{v\})\le |V|-\dim N$
  and  $|V\setminus\{v\}|-\dim (N\contract\{v\})\le |V|-\dim N$.
  Since $M$ is a minor of either $N\delete\{v\}$ or $N\contract
  \{v\}$,
  $|V'|-\dim M\le |V|-\dim N$
  by the induction hypothesis.
\end{proof}
\begin{PROP}
  A minor of a Lagrangian chain-group is Lagrangian.
\end{PROP}
\begin{proof}
  Let $N$ be a Lagrangian chain-group on $V$ to $K$
  and $N'$ be its   minor on $V'$ to $K$.
  By Proposition~\ref{prop:minorisotropic}, $N'$ is  isotropic
  and therefore $\dim(N')\le |V'|$.
  Thus it is enough to show that $\dim(N')\ge |V'|$.
  Since $\dim(N)=|V|$, it follows that $\dim(N')\ge |V'|$ by Corollary~\ref{cor:minordim}.
\end{proof}

\begin{THM}\label{thm:delcondual}
  If $N$ is a chain-group on $V$ to $K$ and $X\subseteq V$, then 
  \[
  (N\delete X)^\bot= N^\bot \delete X
  \text{ and }
  (N\contract X)^\bot= N^\bot \contract X.
  \]
\end{THM}
\begin{proof}
  By symmetry, it is enough to show  that
  $(N\delete X)^\bot= N^\bot \delete X$.
  By induction, we may assume $|X|=1$. 
  Let $v\in X$.
  
  Let $f$ be a chain in $N^\bot\delete X$. 
  There is a chain $f_1\in N^\bot$ such that
  $f_1\cdot (V\setminus X)=f$ and $\form{f_1(v),\tbinom 1 0}_K=0$.
 Let $g\in N$ be a chain such that 
  $\form{g(v),\tbinom 10}_K=0$.
  Then 
 $\form{f_1(v),g(v)}_K=0$ by Lemma~\ref{lem:basic}.
  Therefore $\form{f,g\cdot (V\setminus X)}=\form{f_1,g}=0$ and so 
  $f\in (N\delete X)^\bot$.
  We conclude that $N^\bot \delete X \subseteq (N\delete X)^\bot$.

  We now claim that $\dim (N^\bot\delete X)= \dim (N\delete
  X)^\bot$.
  We apply Proposition~\ref{prop:minordim} to deduce that 
  \begin{align*}
    \dim(N\delete X)-\dim(N)
    &=
    \begin{cases}
      0 &\text{if }\dup v\notin N,
      \dup  v\in N^\bot,\\
      -2 & \text{if }\dup v\in N,
      \dup  v\notin N^\bot,\\
      -1 & \text{otherwise,}
    \end{cases}\\
  \dim(N^\bot\delete X)-\dim(N^\bot)
  &=
  \begin{cases}
    0 &\text{if }\dup v\notin N^\bot,
    \dup  v\in N,\\
    -2 & \text{if }\dup v\in N^\bot,
    \dup  v\notin N,\\
    -1 & \text{otherwise.}
  \end{cases}
  \end{align*}
  By summing these equations, we obtain the following:
  \[
  \dim(N\delete X)-\dim(N)+\dim(N^\bot\delete X)-\dim(N^\bot)=-2.
  \]
  Since $\dim(N)+\dim(N^\bot)=2|V|$ and 
  $\dim(N\delete X)+\dim(N\delete X)^\bot=2(|V|-1)$, 
  we deduce that $\dim(N^\bot\delete X)=\dim(N\delete X)^\bot$.

  Since $N^\bot\delete X\subseteq (N\delete X)^\bot$ and
  $\dim(N^\bot\delete X)=\dim(N\delete X)^\bot$, we conclude that 
  $N^\bot\delete X=(N\delete X)^\bot$.
\end{proof}

\subsection{Connectivity}
We define the connectivity of a chain-group.
Later it will be shown that this definition is related to the
connectivity function of matroids (Lemma~\ref{lem:matroidconn})
and rank functions of matrices (Theorem~\ref{thm:connectivity}).

Let $N$ be a chain-group on $V$ to $K$. 
If $U$ is a subset of $V$, then we write 
\[
\lambda_N(U)=\frac{\dim N-\dim(N\times(V\setminus U))-\dim(N\times U)}{2}.
\]
This function $\lambda_N$ is called the \emph{connectivity function} of a
chain-group~$N$.
By Lemma~\ref{lem:linalg}, we can rewrite $\lambda_N$ as follows:
\[
\lambda_N(U)=\frac{\dim(N\cdot U)-\dim(N\times U)}{2}.
\]
From Theorem~\ref{thm:algdual}, it is easy to derive that 
$
\lambda_{N^\bot}(U)=\lambda_N (U)$.

In general $\lambda_N(X)$ need not be an integer. But if $N$ is
Lagrangian, 
then $\lambda_N(X)$ is always an integer by the following lemma.
\begin{LEM}\label{lem:even}
  If $N$ is a Lagrangian chain-group on $V$ to $K$, then 
  \[\lambda_N(X)=|X|-\dim(N\times X) \]
  for all $X\subseteq V$.
\end{LEM}
\begin{proof}
  From the definition of $\lambda_N(X)$,
  \begin{align*}
    2\lambda_N(X)&=\dim (N\cdot X)-\dim(N\times X)\\
    &=2|X|-\dim (N\cdot X)^\bot -\dim(N\times X)\\
    &=2|X|-\dim (N^\bot\times X)-\dim(N\times X),\\
    \intertext{and since $N=N^\bot$, we have}
    &=2(|X|-\dim(N\times X)).\qedhere
    \end{align*}
\end{proof}
By definition, it is easy to see that
$\lambda_N(U)=\lambda_N(V\setminus U)$.
 Thus $\lambda_N$ is symmetric.
We prove that $\lambda_N$ is submodular.

\begin{LEM}\label{lem:submodular}
  Let  $N$ be a chain-group on $V$ to $K$ and 
  $X$, $Y$ be two subsets of $V$. Then,
  \[
  \dim (N\times (X\cup Y))+\dim (N\times (X\cap Y))\ge 
  \dim (N\times X) +\dim (N\times Y ).
  \]
\end{LEM}
\begin{proof}
  For $T\subseteq V$, let $N_T=\{f\in N:f(v)=0 \text{ for all }v\notin
  T\}$.
  Let $N_X+N_Y=\{f+g:f\in N_X,g\in N_Y\}$.
  We know that 
  $\dim (N_X+N_Y)+\dim(N_X\cap N_Y)=\dim N_X+\dim N_Y$
  from a standard theorem in the linear algebra.
  Since $N_X\cap N_Y=N_{X\cap Y}$ and $N_X+N_Y\subseteq N_{X\cup Y}$,
  we deduce that 
  \[
  \dim N_{X\cup Y}+\dim N_{X\cap Y}\ge \dim N_X+\dim N_Y.
  \]
  Since $\dim N_T=\dim (N\times T)$, we are done. 
\end{proof}
\begin{THM}[Submodular inequality]\label{thm:submodular}
  Let $N$ be a chain-group on $V$ to $K$. Then 
  $\lambda_N$ is submodular; in other words, 
  \[
  \lambda_N(X)+\lambda_N(Y)\ge \lambda_N(X\cup Y)+\lambda_N(X\cap Y)
  \]
  for all $X,Y\subseteq V$.
\end{THM}
\begin{proof}
  We use Lemma~\ref{lem:submodular}.
  Let $S=V\setminus X$ and $T=V\setminus Y$.
  \begin{align*}
    \lefteqn{2\lambda_N(X)+2\lambda_N(Y)}\\
    &=2\dim(N)\\
    &\quad-(\dim(N\times X)+\dim(N\times S)+\dim(N\times
    Y)+\dim(N\times T))\\
    &\ge 2\dim(N)-\dim(N\times (X\cup Y))-\dim(N\times(X\cap Y))\\
    &\quad
    -\dim(N\times (S\cap Y))-\dim(N\times (S\cup Y))\\
    &=2\lambda_N(X\cup Y)+2\lambda_N(X\cap Y).\qedhere
  \end{align*}
\end{proof}
What happens to the connectivity functions if we take minors of a
chain-group?
As in the matroid theory, the connectivity does not increase.
\begin{THM}\label{thm:connectivityminor}
  Let $N$, $M$ be chain-groups on $V$, $V'$ respectively.
  If $M$ is a minor of a chain-group $N$,
  then $\lambda_M(T)\le \lambda_N(T\cup U)$ for all $T\subseteq V'$
  and all $U\subseteq V\setminus V'$.
\end{THM}
\begin{proof}
  By induction on $|V\setminus V'|$, 
  it is enough to prove this when $|V\setminus V'|=1$.
  Let $v\in V\setminus V'$.
  By symmetry we may assume that $M=N\delete \{v\}$.

  We claim that 
  $\lambda_M(T)\le \lambda_N(T)$.
  From the definition, we deduce 
  \begin{align*}
    {2\lambda_M(T)-2\lambda_N (T)}
    &= \dim(N\delete \{v\}\cdot T)
    -\dim(N\delete \{v\}\times T)\\
    &\quad -\dim(N\cdot T)+\dim(N\times T).
 \end{align*}
  Clearly $N\delete \{v\} \cdot T\subseteq N\cdot T$
  and $N\times T\subseteq N\delete\{v\}\times T$.
  Thus $\lambda_M(T)\le \lambda_N(T)$.




  Since $\lambda_N$ and $\lambda_M$ are symmetric, 
  $\lambda_M(T)=\lambda_M(V'\setminus T)\le \lambda_N(V'\setminus
  T)=\lambda_N(T\cup \{v\})$.
\end{proof}

\subsection{Branch-width}
A \emph{branch-decomposition} of a chain-group $N$ on $V$ to $K$
is a pair $(T,\mathcal L)$ of a subcubic tree $T$ and
a bijection $\mathcal L:V\rightarrow \{t: \text{$t$ is a leaf of
  $T$}\}$.
For each edge $e=uv$ of the tree $T$, the connected components of
$T\setminus e$ form a partition $(X_e,Y_e)$ of the leaves of $T$
and we call $\lambda_N (\mathcal L^{-1}(X_e))$ the \emph{width} of~$e$.
The \emph{width} of a branch-decomposition $(T,\mathcal L)$ is the
maximum width of all edges of~$T$.
The \emph{branch-width} $\bw(N)$ of a chain-group $N$ 
is the minimum width of all its branch-decompositions.
(If $|V|\le 1$, then we define that $\bw(N)=0$.)

\section{Matrix Representations of Lagrangian Chain-groups}\label{sec:matrix}
\subsection{Matrix Representations.}
We say that  two  chains $f$ and $g$ on $V$ to $K$ are
\emph{supplementary} if, for all $x\in V$, 
\begin{enumerate}[(i)]
\item $  \form{f(x),f(x)}_K=\form{g(x),g(x)}_K=0 $ and 
\item  $ \form{f(x),g(x)}_K=1 $.
\end{enumerate}
Given a skew-symmetric or symmetric matrix $A$, we may construct a
Lagrangian chain-group as follows.

\begin{PROP}\label{prop:matrix2chain}
  Let $M=(m_{ij}:i,j\in V)$ be a skew-symmetric or symmetric $V\times V$ 
  matrix over a field $\F$.
  Let $a,b$ be supplementary chains on $V$ to $K=\F^2$
  where
  $\form{\,,\,}_K$ is skew-symmetric if $M$ is symmetric and
  symmetric if $M$ is skew-symmetric.

  For $i\in V$, let $f_i$ be a chain on $V$ to $K$ such that
  for all $j\in V$, 
  \[  f_i(j)=
  \begin{cases}
  m_{ij}a(j)+b(j) &\text{if }j=i, \\
  m_{ij}a(j) &\text{if } j\neq i.
  \end{cases}
  \]
  Then the subspace $N$ of $K^V$ spanned by chains $\{f_i:i\in
  V\}$ is 
  a Lagrangian chain-group on $V$ to $K$.
\end{PROP}
If $M$ is a skew-symmetric or symmetric matrix and $a$, $b$ are
supplementary chains on $V$ to $K$, then 
we call  $(M,a,b)$ a \emph{(general) matrix representation} of a Lagrangian
chain-group~$N$.
Furthermore if $a(v),b(v)\in \{\pm\tbinom10,\pm\tbinom01\}$ for each
$v\in V$, then $(M,a,b)$ is called a \emph{special matrix
  representation} of~$N$.

\begin{proof}
  For all $i\in V$,
  \[
  \form{f_i,f_i}=\sum_{j\in V}\form{f_i(j),f_i(j)}_K
  = m_{ii} (\form{a(i),b(i)}_K+\form{b(i),a(i)}_K)
  =0,
  \]
  because either $m_{ii}=0$ (if $M$ is skew-symmetric) or $\form{\,,\,}_K$ is skew-symmetric.

  Now let $i$ and $j$ be two distinct elements of $V$.  Then,
  \begin{align*}
  \form{f_i,f_j}
  &=\form{f_i(i),f_j(i)}_K
  +\form{f_i(j),f_j(j)}_K\\
  &=m_{ji} \form{b(i),a(i)}_K
  +m_{ij} \form{a(j),b(j)}_K\\
  &=0,
  \end{align*}
  because either $m_{ij}=-m_{ji}$ and
  $\form{b(i),a(i)}_K=\form{a(j),b(j)}_K$
  or $m_{ij}=m_{ji}$ and
  $\form{b(i),a(i)}_K=-\form{a(j),b(j)}_K$.

  It is easy to see that $\{f_i:i\in V\}$ is linearly independent and
  therefore $\dim(N)= |V|$. 
  This proves that $N$ is a Lagrangian chain-group.
  %
  %
\end{proof}

\subsection{Eulerian chains.}
A chain $a$ on $V$ to $K$ is called a \emph{(general) eulerian} chain of 
an
isotropic chain-group $N$
if
\begin{enumerate}[(i)]
\item $a(x)\neq 0$, $\form{a(x),a(x)}_K=0$ for all $x\in V$
and 
\item  there is no  non-zero chain $f\in N$ such
that $\form{f(x),a(x)}_K=0$ for all $x\in V$.
\end{enumerate}
A general eulerian chain $a$ is a \emph{special eulerian} chain 
if 
for all $v\in V$, $a(v)\in \{\pm\tbinom10,\pm\tbinom01\}$.
It is easy to observe that if $(M,a,b)$ is a general (special) matrix representation of
a Lagrangian chain-group $N$, then $a$ is a general (special) eulerian chain of $N$.
We will prove that every 
general eulerian chain of a Lagrangian chain-group
induces a matrix representation.
Before proving that, we first show that every Lagrangian chain-group
has a special eulerian chain.

\begin{PROP}
  Every isotropic chain-group has a special eulerian chain.
\end{PROP}
\begin{proof}
  Let $N$ be an isotropic chain-group on $V$ to $K=\F^2$.
  We proceed by induction on $|V|$.
  We may assume that $\dim(N)>0$.
  Let $v\in V$.

  If $|V|=1$, then $\dim(N)=1$.  
  Then either $\dup v$ or $\ddown v$ is a special eulerian chain.

  Now let us assume that $|V|>1$. 
  Let $W=V\setminus \{v\}$.
 Both $N\delete \{v\}$ and $N\contract
  \{v\}$ 
  are isotropic chain-groups on $W$ to $K$. 
  By the induction hypothesis, 
  both $N\delete \{v\}$ and $N\contract \{v\}$ 
  have special eulerian chains $a_1'$, $a_2'$,
  respectively, on $W$ to $K$
  such that $a_i'(x)\in \{\tbinom10,\tbinom01\}$ for all $x\in W$.

  Let $a_1$, $a_2$ be chains on $V$ to $K$
  such that   $a_1(v)=\tbinom10$, $a_2(v)=\tbinom01$,
  and 
  $a_i\cdot W=a_i'$ for $i=1,2$.
  We claim that either $a_1$ or $a_2$ is a special eulerian chain of $N$.
  Suppose not. For each $i=1,2$, there is a nonzero chain $f_i\in N$ such that 
  $\form{f_i(x),a_i(x)}_K=0$ for all $x\in V$.
  By construction $f_1\cdot W\in N\delete \{v\}$ 
  and $f_2\cdot W\in N\contract \{v\}$.
  Since $a_1'$, $a_2'$ are special eulerian chains of $N\delete \{v\}$
  and $N\contract\{v\}$, respectively, 
  we have $f_1\cdot W=f_2\cdot W=0$.
  
  Since $f_i\neq 0$, by Lemma~\ref{lem:basic}, 
  $f_1=c_1 \dup v$ and $f_2=c_2\ddown v$ for some nonzero $c_1,c_2\in
  \F$.
  Then $\form{f_1,f_2}=\form{f_1(v),f_2(v)}_K=c_1c_2\neq 0$, contradictory to the assumption that $N$ is isotropic.
\end{proof}

\begin{PROP}\label{prop:chain2matrix}
  Let $N$ be a Lagrangian chain-group on $V$ to $K$
  and let $a$ be a general eulerian chain of $N$
  and let $b$ be a chain supplementary to $a$.

\begin{enumerate}
\item 
  For every $v\in V$, there exists a unique chain $f_v\in N$
  satisfying the following two conditions.
  \begin{enumerate}[(i)]
  \item $\form{a(v),f_v(v)}_K=1$,
  \item $\form{a(w),f_v(w)}_K=0$ for all $w\in V\setminus\{v\}$.
  \end{enumerate}
  Moreover, $\{f_v:v\in V\}$ is a basis of $N$. This basis is called
  the \emph{fundamental basis} of $N$ with respect to $a$.

\item

  If  $\form{\,,\,}_K$ is symmetric and 
  either the characteristic of $\F$ is not $2$ or 
  $f_v(v)=b(v)$ for all $v\in V$,
  then
  $M=(\form{f_i(j),b(j)}_K:i,j\in V)$ 
  is a  skew-symmetric matrix such that $(M,a,b)$ is a
  general matrix
  representation of~$N$.
\item 
  If $\form{\,,\,}_K$ is skew-symmetric, 
  $M=(\form{f_i(j),b(j)}_K:i,j\in V)$ 
  is a symmetric matrix such that $(M,a,b)$ is a general matrix representation of~$N$.
\end{enumerate}
\end{PROP}
\begin{proof}
  Existence in (1):
  For each $x\in V$, let $g_x$ be a chain on $V$ to $K$ such that 
  $g_x(x)=a(x)$ and $g_x(y)=0$ for all $y\in V\setminus \{x\}$.
  Let $W$ be a chain-group spanned by $\{g_x:x\in V\}$. 
  It is clear that $\dim(W)=|V|$.
  Let $N+W=\{f+g:f\in N,g\in W\}$. Since $a$ is eulerian,
  $N\cap W=\{0\}$ and therefore $\dim(N+W)=\dim(N)+\dim(W)=2|V|$,
  because $N$ is Lagrangian.
  We conclude that $N+W=K^V$.
  Let $h_v$ be a chain on $V$ to $K$ such that
  $\form{a(v),h_v(v)}_K=1$ and $h_v(w)=0$ for all $w\in V\setminus\{ v\}$.
  We express $h_v=f_v+g$ for some $f_v\in N$ and $g\in W$.
  Then 
  $\form{a(v),f_v(v)}_K=\form{a(v),h_v(v)}_K-\form{a(v),g(v)}_K=1$
  and 
  $\form{a(w),f_v(w)}_K=\form{a(w),h_v(w)}_K-\form{a(w),g(w)}_K=0$
  for all $w\in V\setminus \{v\}$.
  
  \smallskip

  Uniqueness in (1):
  Suppose that there are two chains $f_v$ and $f_v'$ in $N$ satisfying two
  conditions (i), (ii) in (1).
  Then
  $\form{a(v),f_v(v)-f_v'(v)}_K=0$.
  By Lemma~\ref{lem:basic}, 
  there exists $c\in \F$ such that
  $f_v(v)-f_v'(v)=c a(v)$.
  Let  $f=f_v-f_v'\in N$. Then $\form{a(w),f(w)}_K=0$ for all $w\in
  V$.
  Since $a$ is eulerian,
  $f=0$ and therefore $f_v=f_v'$.

  \smallskip

  Being a basis in (1):  
  We claim that $\{f_v:v\in V\}$ is linearly independent. Suppose that
  $\sum_{w\in V} c_w f_w=0$ for some $c_w\in \F$.
  Then
  $  c_v
  =\sum_{w\in V} c_w\form{a(v), f_w(v)}_K=0$
 for all $v\in V$.
 
 \smallskip

 Constructing a matrix for (2) and (3):
  Let $i,j\in V$.
   By (ii) and Lemma~\ref{lem:basic}, 
   there exists $m_{ij}\in \F$ such that 
  $f_i(j)=m_{ij}a(j)$ if $i\neq j$
  and 
  $f_i(i)-b(i)=m_{ii}a(i)$.
  Then, $\form{f_i(j),b(j)}_K=m_{ij}$ for all $i,j\in V$.
  Therefore $M=(m_{ij}:i,j\in V)$. 
  
  Since $N$ is isotropic, 
  \[\form{f_i,f_j}=\sum_{v\in V} \form{f_i(v),f_j(v)}_K=0\] and we deduce that
  $\form{f_i(i),f_j(i)}_K+\form{f_i(j),f_j(j)}_K=0$ if $i\neq j$
  and 
  $\form{f_i(i),f_i(i)}_K=0$.
  This implies that 
  \[
  m_{ji} \form{b(i),a(i)}_K+m_{ij}\form{a(j),b(j)}_K=0 \text{ for all
  }i,j\in V.  
  \]

  If $\form{\,,\,}_K$ is skew-symmetric, then $\form{b(i),a(i)}_K=-1$
  and therefore $m_{ji}=m_{ij}$.

  If $\form{\,,\,}_K$ is symmetric, then $\form{b(i),a(i)}_K=1$ and so
  $m_{ji}=-m_{ij}$. 
  This also imply that $m_{ii}=0$ if the characteristic of $\F$ is not
  $2$.
  If the characteristic of $\F$ is $2$, then we assumed that
  $f_i(i)=b(i)$ and therefore $m_{ii}=0$.
  Note that $\form{f_i(i),f_i(i)}_K=0$ and therefore the chain $b$
  with $b(i)=f_i(i)$ for all $i\in V$ is supplementary to $a$.
  %
  
  It is easy to observe that $(M,a,b)$  is a general matrix representation of
  $N$
  because $a$, $b$ are supplementary
  and $f_i(j)=m_{ij}a(j)+b(j)$ if $i=j\in V$
and $f_i(j)=m_{ij}a(j)$ if $i\neq j$.
\end{proof}

\begin{PROP}\label{prop:eulerian}
  Let $(M,a,b)$ be a special matrix representation of a Lagrangian
  chain-group $N$ on $V$ to $K=\F^2$. 
  Suppose that $a'$ is a chain such that $a'(v)\in
  \{\pm\tbinom10,\pm\tbinom01\}$
  for all $v\in V$.
  Then $a'$ is special eulerian if and only if 
  $M[Y]$  is nonsingular for $Y=\{x\in V: a'(x)\neq \pm a(x) \}$.
\end{PROP}
\begin{proof}
  Let $M=(m_{ij}:i,j\in V)$.
  Let $f_i\in N$ be a chain such that $f_i(j)=m_{ij}a(j)$ if $j\neq i$ and $f_i(i)=m_{ii}a(i)+b(i)$.

  We first prove that if $M[Y]$ is nonsingular, then $f$ is special
  eulerian. 
  Suppose that there is a chain $f\in N$ such that
  $\form{f(x),a'(x)}_K=0$ for all $x\in V$.
  We may express $f$ as a linear combination $\sum_{i\in V} c_i f_i$ with some $c_i\in \F$.
  If $j\notin Y$, then $a'(j)=\pm a(j)$
  and $\form{f(j),a(j)}_K = c_j \form{b(j),a(j)}_K=0$ and therefore
  $c_j=0$ for all $j\notin Y$.

  If $j\in Y$, then $a'(j)=\pm b(j)$ and so 
  \[
  \form{f(j),b(j)}_K= \sum_{i\in Y} c_i m_{ij}\form{a(j),b(j)}_K=
  \sum_{i\in Y} c_i m_{ij} =
  0.\]
  Since $M[Y]$ is invertible, the only solution $\{c_i:i\in Y\}$
  satisfying the above linear equation is zero. So $c_i=0$ for all
  $i\in V$ and therefore $f=0$, meaning that $a'$  is special
  eulerian.

  Conversely 
  suppose that $M[Y]$ is singular. 
  Then there is a linear combination of rows in $M[Y]$ whose sum is zero.
  Thus there is a non-zero linear combination $\sum_{i\in Y} c_i f_i$
  such that
  \[\form{\sum_{i\in Y} c_i f_i(x),b(x)}_K=0\text{ for all }x\in Y.\]
  Clearly 
  $\form{\sum_{i\in Y} c_i f_i(x),a(x)}_K=0$ for all $x\notin Y$.
  Since at least one $c_i$ is non-zero, $\sum_{i\in Y}c_i f_i$ is
  non-zero. 
  Therefore $a'$ can not be special eulerian. 
\end{proof}


For a subset $Y$ of $V$, 
let $I_Y$ be a $V\times V$ \emph{indicator diagonal matrix} such that 
each diagonal entry corresponding to $Y$ is $-1$
and all other diagonal entries are $1$.
\begin{PROP}\label{prop:chainpivot}
  Suppose that $(M,a,b)$ is a special matrix representation of a Lagrangian
  chain-group $N$
  on $V$ to $K=\F^2$.
  Let $Y\subseteq V$.
  Assume that $M[Y]$ is nonsingular.
  \begin{enumerate}
  \item If $\form{\,,\,}_K$ is symmetric, then 
    $(M*Y,a',b')$ is another special matrix representation of $N$
    where $M*Y$ is skew-symmetric and 
    \[
      a'(v)=
      \begin{cases}
        a(v) &\text{if }v\notin Y,\\
        b(v) &\text{otherwise,}
      \end{cases}
      \quad\quad
      b'(v)=
      \begin{cases}
        b(v) &\text{if }v\notin Y,\\
        a(v) &\text{otherwise.}
      \end{cases}
      \]
    \item If $\form{\,,\,}_K$ is skew-symmetric, then 
      $(I_Y (M*Y),a',b')$ is  another special matrix representation of $N$
      where $I_Y(M*Y)$ is symmetric and 
      \[
      a'(v)=
      \begin{cases}
        a(v) &\text{if }v\notin Y,\\
        b(v) &\text{otherwise,}
      \end{cases}
      \quad\quad
      b'(v)=
      \begin{cases}
        b(v) &\text{if }v\notin Y,\\
        -a(v) &\text{otherwise.}
      \end{cases}
      \]
 \end{enumerate}
\end{PROP}
\begin{proof}
  Let $M=(m_{ij}:i,j\in V)$.
  For each $i\in V$, let $f_i\in N$ be a chain such that
  $f_i(j)=m_{ij}a(j)$ if $j\neq  i$
  and $f_i(i)=m_{ij}a(j)+b(j)$ if $j=i$.
  Since $(M,a,b)$ is a special matrix representation of $N$, 
  $\{f_i:i\in V\}$ is a fundamental basis of $N$.

  Proposition~\ref{prop:eulerian} implies that $a'$ is eulerian.
  %
 According to Proposition~\ref{prop:chain2matrix}, we should be able
  to construct a special matrix representation with respect to the eulerian
  chain $a'$. To do so, we first construct the fundamental basis 
  $\{g_v: v\in V\}$ of $N$
  with respect to $a'$.

  Suppose that for each $x\in V$, 
  $g_x=\sum_{i\in V}c_{xi} f_i$  for some $c_{xi}\in \F$.
  By definition, 
  $\form{a'(x),g_x(x)}_K=1$ and 
  $\form{a'(j),g_x(j)}_K=0$  for all $j\neq x$.
  Then
  \[
  \form{a'(j),g_x(j)}_K=
  \begin{cases}
    \sum_{i\in V} c_{xi} m_{ij} \form{b(j),a(j)}_K, 
    & \text{if }j\in Y,\\
    c_{xj}.
    & \text{if }j\notin Y.
  \end{cases}
  \]

  Suppose that $x\in Y$.
  If $j\in Y$, then 
  \[\sum_{i\in Y} c_{xi} m_{ij} \form{b(j),a(j)}_K=
  \begin{cases}
    1 & \text{if } x=j,\\
    0 &\text{if }x\neq j.
  \end{cases}
  \]
  Let $(m_{ij}':i,j\in Y)=(M[Y])^{-1}$.
  Then $c_{xi}$ is given by the row of $x$ in  $(M[Y])^{-1}$; in other
  words,
  if $x,i\in Y$, then
  $c_{xi}=m_{xi}'$ if $\form{\,,\,}_K$ is symmetric
  and $c_{xi}=-m_{xi}'$ otherwise.
  If $x\in Y$ and $i\notin Y$, then $c_{xi}=0$.

  If $x\notin Y$,  then clearly $c_{xx}=1$ and $c_{xi}=0$ for all
  $i\in V\setminus (Y\cup \{x\})$.
  If $j\in Y$, then 
  $ \sum_{i\in Y} c_{xi}m_{ij} \form{b(j),a(j)}_K
  +c_{xx}m_{xj}\form{b(j),a(j)}_K=0$ 
  and therefore 
  $\sum_{i\in Y} c_{xi} m_{ij} =-m_{xj}$.
  For each $k$ in $Y$, we have
  $c_{xk}=\sum_{i\in Y} c_{xi} \sum_{j\in Y} m_{ij} m'_{jk} 
  = \sum_{j\in Y} m'_{jk} \sum_{i\in Y} c_{xi} m_{ij}
  =-\sum_{j\in Y} m'_{jk}m_{xj}$
  and therefore  for $x\notin Y$ and $i\in Y$, 
  $c_{xi}=-\sum_{j\in Y} m_{xj} m'_{ji}$ 
 
  We determined  the fundamental basis $\{g_x:x\in V\}$ with
  respect to $a'$. We now wish to compute the matrix according to
  Proposition~\ref{prop:chain2matrix}.
 Let us compute $\form{g_x(y),b'(y)}_K$.

  If $x,y\in Y$, then 
  \begin{multline*}
  \form{\sum_{i\in Y} c_{xi} f_i(y), b'(y)}_K\\= 
   c_{xy} \form{b(y), b'(y)}_K
   =c_{xy}=
   \begin{cases}
     m'_{xy} &\text{if  $ \form{\,,\,}_K$ is symmetric,}\\
     -m'_{xy} & \text{if $\form{\,,\,}_K$ is skew-symmetric.}
   \end{cases}
   \end{multline*}
   If $x\in Y$ and $y\notin Y$, then 
  \begin{multline*}
    \form{\sum_{i\in Y} c_{xi} f_i(y), b'(y)}_K= 
    \sum_{i\in Y} c_{xi} m_{iy} \form{a(y), b(y)}_K\\
    =
    \begin{cases}
      \sum_{i\in Y} m'_{xi} m_{iy}.
      &\text{if  $ \form{\,,\,}_K$ is symmetric,}\\
      -\sum_{i\in Y} m'_{xi} m_{iy}.
     & \text{if $\form{\,,\,}_K$ is skew-symmetric.}
   \end{cases}
 \end{multline*}
  If $x\notin Y$ and $y\in Y$, then 
  \[
  \form{\sum_{i\in Y} c_{xi} f_i(y)+f_x(y), b'(y)}_K= c_{xy}=
  -\sum_{j\in Y} m_{xj} m'_{jy}.
  \]
  If $x\notin Y$ and $y\notin Y$, then 
  \[
  \form{\sum_{i\in Y}c_{xi}f_i(y)+f_x(y),b'(y)}_K
  =
  -\sum_{i,j\in Y}m_{xj}m'_{ji} m_{iy} + m_{xy}
  \]

  If $\form{\,,\,}_K$ is symmetric and the characteristic of $\F$ is
  $2$, then we need to ensure that $M$ has no non-zero diagonal
  entries 
  by verifying the additional assumption in (2) of
  Proposition~\ref{prop:chain2matrix} asking that
  $b'(x)=g_x(x)$ for all $x\in V$.
  It is enough to show that 
  \[
  \form{g_x(x),b'(x)}_K=0\text{ for all }x\in V,
  \]
 because, if so, then   $\form{a'(x),b'(x)}_K=1=\form{a'(x),g_x(x)}_K$ implies that
  $g_x(x)=b'(x)$.
  Since $M[Y]$ is skew-symmetric, so is its inverse and therefore
  $m'_{xx}=0$ for all $x\in Y$.
  Furthermore, for each $i,j\in Y$ and $x\in V\setminus Y$,
  we have $m_{xj} m'_{ji} m_{ix} =-m_{xi}m'_{ij}m_{jx}$ 
  because $M$ and $(M[Y])^{-1}$ are skew-symmetric
  and therefore
  $\sum_{i,j\in Y} m_{xj}m'_{ji}m_{ix}=0$. Thus 
  $g_x(x)=b'(x)$ for all $x\in V$ if $\form{\,,\,}_K$ is symmetric and the
characteristic of $\F$ is $2$.


  We conclude that the matrix $(\form{g_i(j),b'(j)}_K:i,j\in V)$
  is indeed $M*Y$ if $\form{\,,\,}_K$ is symmetric
or $(I_Y) (M*Y)$ if $\form{\,,\,}_K$ is skew-symmetric.
This concludes the proof.
\end{proof}


A matrix $M$ is called a \emph{fundamental matrix} of a Lagrangian chain-group $N$
if $(M,a,b)$ is a special matrix representation of $N$
for some chains $a$ and $b$.
We aim to characterize when two matrices $M$ and $M'$ 
are fundamental matrices of the same Lagrangian chain-group.

\begin{THM}\label{thm:equivmatrix}
  Let $M$ and $M'$ be $V\times V$ skew-symmetric or symmetric matrices
  over $\F$.
  The following are equivalent.
  \begin{enumerate}[(i)]
  \item There is a Lagrangian chain-group $N$
    such that
    both $(M,a,b)$ and $(M',a',b')$ are special matrix
    representations of $N$ for some chains $a$, $a'$, $b$, $b'$.
  \item There is $Y\subseteq V$ such that $M[Y]$ is nonsingular
    and 
    \[M'=
    \begin{cases}
      D(M*Y)D & \text{if $\form{\,,\,}_K$ is symmetric,}\\
      D I_Y(M*Y)D &\text{if $\form{\,,\,}_K$ is skew-symmetric}
   \end{cases}
   \]
    for some diagonal matrix $D$ whose diagonal entries
    are $\pm1$.
 \end{enumerate}
\end{THM}

\begin{proof}
  To prove (i) from (ii), we use 
  Proposition~\ref{prop:chainpivot}.
  Let $a(v)=\tbinom10$ and $b(v)=\tbinom01$ for all $v\in V$.
  Let $N$ be the Lagrangian chain-group with the special matrix representation $(M,a,b)$.
  Let $M_0=M*Y$ if $\form{\,,\,}_K$ is symmetric
  and $M_0=I_Y (M*Y)$ if $\form{\,,\,}_K$ is skew-symmetric.
  By Proposition~\ref{prop:chainpivot}, 
  there are chains $a_0$, $b_0$ so that 
  $(M_0,a_0,b_0)$ is a special matrix representation of $N$.
  Let $Z$ be a subset of $V$ such that $I_Z=D$.
  For each $v\in V$, let
  \[
  a'(v)=
  \begin{cases}
    -a_0(v) &\text{if }v\in Z,\\
    a_0(v) &\text{if }v\notin Z,
  \end{cases}
  \quad
  b'(v)=
  \begin{cases}
    -b_0(v) &\text{if }v\in Z,\\
    b_0(v) &\text{if }v\notin Z.
  \end{cases}
  \]
  Then $a'$, $b'$ are supplementary and
  $(M',a',b')$ is a special matrix representation of $N$
  because $M'=DM_0D$.

  Now let us assume (i) and prove (ii).
  Let $Y=\{x\in V: a'(x)\neq\pm a(x)\}$.
  Since $a'$ is a special eulerian chain
  of $N$,   
  $M[Y]$ is nonsingular
  by Proposition~\ref{prop:eulerian}.
  By replacing $M$ with $M*Y$ if $\form{\,,\,}_K$ is symmetric,
  or $I_Y(M*Y)$ if $\form{\,,\,}_K$ is skew-symmetric, 
  we may assume that $Y=\emptyset$.
  Thus $a'(x)=\pm a(x)$ and $b'(x)=\pm b(x)$ for all $x\in V$.
  Let $Z=\{x\in V: a'(x)=-a(x)\}$ and $D=I_Z$.
  Since $\form{a'(x),b'(x)}_K=1$, $b'(x)=-b(x)$ if and only if $x\in
  Z$.
  Then $(DMD,a',b')$ is a special matrix representation of $N$,
  because the fundamental basis generated by $(DMD,a',b')$ spans the
  same subspace $N$ spanned by the fundamental basis generated by
  $(M,a,b)$.
  We now have two special matrix representations $(M',a',b')$ and $(DMD,a',b')$.
  By Proposition~\ref{prop:chain2matrix}, $M'=DMD$ because of the
  uniqueness of the fundamental basis with respect to $a'$.
  This concludes the proof.
\end{proof}
\emph{Negating} a row or a column of a matrix 
is to multiply $-1$ to each of its entries.
Obviously a matrix obtained by negating some rows and columns of a $V\times V$ matrix $M$
is of the form $I_X M I_Y$ for some $X,Y\subseteq V$.
We now prove that the order of applying pivots and negations can be
reversed.
\begin{LEM}\label{lem:negating}
  Let $M$ be a $V\times V$ matrix and let $Y$ be a subset of $V$ such
  that $M[Y]$ is nonsingular.
Let $M'$ be  a matrix obtained from $M$ by negativing some rows and columns.
  Then $M'*Y$
  can be obtained from $M*Y$ by negating some rows and columns. (See Figure~\ref{fig:dig}.)
\end{LEM}
\begin{figure}
  \centering
  \label{fig:dig}
  \[
  \begin{CD}
    M @>\text{pivot}>> M*Y\\
    @V \substack{\text{negating some}\\\text{rows and columns}} VV   
    @VV\substack{\text{negating some}\\\text{rows and columns}} V   \\
   M' @>\text{pivot}>> M'*Y
  \end{CD}
  \]
  \caption{Commuting pivots and negations}
\end{figure}
\begin{proof}
 More generally we write $M$ and $M'$ as follows:
  \[
  M=\bordermatrix{ & Y & V\setminus Y\cr
    Y  & A & B \cr
    V\setminus Y & C & D}
 ,\quad
 M'=\bordermatrix{ & Y & V\setminus Y\cr
    Y  & JAK & JBL \cr
    V\setminus Y & UCK & UDL},
 \]
 for some nonsingular diagonal matrices $J$, $K$, $L$, $U$.
 Then 
 \begin{align*}
   M*Y&=\begin{pmatrix}
     A^{-1} & A^{-1}B \\
     -CA^{-1} & D-CA^{-1}B
 \end{pmatrix},
 \\
 M'*Y&=\begin{pmatrix}
    K^{-1}A^{-1} J^{-1} & K^{-1}A^{-1} J^{-1} JBL \\
   -UCK K^{-1}A^{-1}J^{-1} & UDL-UCK K^{-1}A^{-1}
   J^{-1} JBL 
 \end{pmatrix}\\
 &=
 \begin{pmatrix}
    K^{-1}(A^{-1}) J^{-1} & K^{-1}(A^{-1}B)L \\
    U(-CA^{-1})J^{-1} & U(D-CA^{-1}B)L 
   \end{pmatrix}.
 \end{align*}
 This lemma follows because we can set $J$, $K$, $L$, $U$ to be
 diagonal matrices with $\pm1$ on the diagonal entries
 and then $M'*Y$ can be obtained from $M*Y$ by negating some rows and columns.
\end{proof}

\subsection{Minors.}
Suppose that $(M,a,b)$ is a special matrix representation of a Lagrangian
chain-group $N$. We will find special matrix representations of minors of $N$.

\begin{LEM}\label{lem:minor}
Let $(M,a,b)$ be a  special matrix representation 
of a Lagrangian chain-group $N$ on $V$ to $K=\F^2$.
Let $v\in V$ and $T=V\setminus \{v\}$.
Suppose that $a(v)=\pm\tbinom10$. 
\begin{enumerate}
\item 
The triple $(M[T],a\cdot T, b\cdot T)$ is 
a special matrix representation of $N\delete \{v\}$.

\item
There is $Y\subseteq V$ such that
$M[Y]$ is nonsingular and 
  $(M'[T],a'\cdot T,b'\cdot T)$ is a special matrix representation of
  $N\contract \{v\}$, where
  \[
  M'=
  \begin{cases}
    M*Y & \text{if $\form{\,,\,}_K$ is symmetric,}\\
    (I_Y)(M*Y) & \text{if $\form{\,,\,}_K$ is skew-symmetric,}
  \end{cases}
  \]
  and $a'$ and $b'$ are given by Proposition~\ref{prop:chainpivot}.

\end{enumerate}
\end{LEM}
\begin{proof}
  Let $M=(m_{ij}:i,j\in V)$
  and for each $i\in V$,
  let $f_i\in N$ be a chain as it is defined in Proposition~\ref{prop:matrix2chain}.

  (1):
  We know that $f_i\cdot T\in N\delete\{v\}$ for all $i\neq v$.
  Since $a$ is eulerian, 
  $\dup v\notin N$
  and therefore $\{f_i\cdot T: i\in T\}$ is linearly
  independent.
  Then $\{f_i\cdot T: i\in T\}$  is a basis of $N\delete
  \{v\}$,
  because $\dim(N\delete \{v\})=|T|=|V|-1$.
  Now it is easy to verify that $(M[T],a\cdot T,b\cdot T)$ is a
  special matrix representation of $N\delete \{v\}$.
  (2):
 If $m_{iv}=m_{vi}=0$ for all $i\in V$, then we may
  simply replace $a(v)$ with $\pm\tbinom01$ and $b(v)$ with $\pm\tbinom10$
  without changing the Lagrangian chain-group $N$. In this case, we
  simply apply (1) to deduce that
  $Y=\emptyset$  works.

  Otherwise, there exists $Y\subseteq V$ such that $v\in Y$ and $M[Y]$
  is nonsingular because $M$ is skew-symmetric or symmetric. 
  We apply $M*Y$ to get $(M',a',b')$ as an alternative
  special matrix representation of $N$ by
  Proposition~\ref{prop:chainpivot}. Then $a'(v)=\pm\tbinom01$ and
  then we apply (1) to $(M',a',b')$.
\end{proof}
\begin{THM}\label{thm:minor}
  For $i=1,2$, let $M_i$ be a fundamental matrix of a Lagrangian
  chain-group $N_i$ on $V_i$ to $K=\F^2$.
  If $N_1$ is simply isomorphic to a minor of $N_2$, 
  then 
  $M_1$ is isomorphic to a principal submatrix of a matrix obtained
  from $M_2$ by 
  taking a pivot and 
  negating some rows and columns.
\end{THM}
\begin{proof}
  Since $K$ is shared by $N_1$ and $N_2$, $M_1$ and $M_2$ are
  skew-symmetric if $\form{\,,\,}_K$ is symmetric
  and symmetric if $\form{\,,\,}_K$ is skew-symmetric.

  We may assume that $N_1$ is a minor of $N_2$
  and $V_1\subseteq V_2$.
  Then by Lemmas~\ref{lem:negating} and \ref{lem:minor}, 
  $N_1$ has a fundamental matrix $M'$
  that is a principal submatrix of a matrix
  obtained from $M$ by taking a pivot and negativing some rows
  if necessary.
  Then both $M'$ and $M_1$ are fundamental matrices of $N_1$.  By
  Theorem~\ref{thm:equivmatrix},  there is a method to get $M_1$ from
  $M'$ by applying a pivot and negating some rows and columns if
  necessary.
  %
\end{proof} 

\subsection{Representable Delta-matroids.}
Theorem~\ref{thm:tucker} implies the following proposition.
\begin{PROP}\label{prop:matdelta}
  Let $A$, $B$ be skew-symmetric or symmetric matrices over  a field $\F$.
  If $A$ is a principal submatrix of a matrix obtained from $B$
  by taking a pivot and negating some rows and columns, then 
  the delta-matroid $\MM(A)$ is a minor of $\MM(B)$.
\end{PROP}

Bouchet~\cite{Bouchet1988b} showed that there is a natural way to 
construct a delta-matroid from 
an isotropic chain-group.
\begin{THM}[Bouchet \cite{Bouchet1988b}]\label{thm:isodelta}
  Let $N$ be an isotropic chain-groups $N$ on $V$ to $K$. 
  Let $a$ and $b$  be supplementary chains on $V$ to $K$.
  Let \begin{align*}\FF=\{X\subseteq V:
    &\text{there is no non-zero chain }f\in N \\
    &\text{such that }
    \form{f(x),a(x)}_K=0 \text{ for all }x\in V\setminus X\\
    &\quad\text{ and }
    \form{f(x),b(x)}_K=0 \text{ for all }x\in X
    .\}
  \end{align*}
  Then, $\MM=(V,\FF)$ is a delta-matroid.
\end{THM}
The triple $(N,a,b)$ given as above is called  the \emph{chain-group representation} of the
delta-matroid~$\MM$.
In addition, if $a(v), b(v)\in \{\pm\tbinom10,\pm\tbinom01\}$, then
$(N,a,b)$ is called the \emph{special chain-group representation} of $\MM$.

We remind you that a delta-matroid $\MM$ is representable over a field $\F$
if $\MM=\MM(A)\Delta Y$ for some skew-symmetric or symmetric 
$V\times V$ matrix
$A$ over $\F$  and a subset $Y$ of $V$
where
$\MM(A)=(V,\FF)$ where $\FF=\{Y: \text{ $A[Y]$ is nonsingular}\}$.

Suppose that $N$ is a Lagrangian chain-group represented by a special matrix
representation $(M,a,b)$.
Then $(N,a,b)$ induces a delta-matroid $\MM$ by the above theorem.
Proposition~\ref{prop:eulerian} characterizes all the special eulerian
chains in terms of the singularity of $M[Y]$
and special eulerian chains coincide with the feasible sets of $\MM$
given by Theorem~\ref{thm:isodelta}.
In other words, $Y$ is feasible in $\MM$ if and only if a chain $a'$
is special eulerian in $N$ when 
$a(v)=a'(v)$ if $v\in Y$ and $a'(v)=b(v)$ if
$v\notin Y$.

Then twisting operations $\MM\Delta Y$ on delta-matroids can be simulated 
by swapping supplementary chains $a(x)$ and $b(x)$ for $x\in Y$
in the chain-group representation as
it is in Proposition~\ref{prop:chainpivot}.
Thus we can alternatively define representable delta-matroids as
follows.
\begin{THM}\label{thm:repdelta}
  A delta-matroid on $V$ is representable over a field~$\F$
  if and only if
  it admits a special chain-group representation $(N,a,b)$ for a Lagrangian
  chain-group $N$  on $V$ to $K=\F^2$
  and special supplementary chains $a$, $b$ on $V$ to $K$
  where $\form{\,,\,}_K$ is either skew-symmetric or symmetric.
\end{THM}

\subsection{Connectivity.}
When the rank-width of matrices is defined, the function $\rk
M[X,V\setminus X]$ is used to describe how complex the connection
between $X$ and $V\setminus X$ is.
In this subsection, we express $\rk M[X,V\setminus X]$ in terms of a Lagrangian chain-group
represented by $M$.
\begin{THM}\label{thm:connectivity}
  Let $M$ be a skew-symmetric or symmetric $V\times V$ matrix over a
  field $\F$.
  Let $N$ be a Lagrangian chain-group on $V$ to $K=\F^2$ such that
  $(M,a,b)$ is a matrix representation of $N$ 
  with supplementary chains $a$ and $b$ on $V$ to $K$.
  Then, 
  \[
  \rk M[X,V\setminus X]=\lambda_N(X)=|X|-\dim(N\times X).
  \]
\end{THM}
\begin{proof}
  Let $M=(m_{ij}:i,j\in V)$. As we described in 
  Proposition~\ref{prop:matrix2chain}, we let $f_i(j)=m_{ij}a(j)$
  if $j\in V\setminus \{i\}$
  and $f_i(i)=m_{ii}+b(i)$.
 We know that $\{f_i:i\in V\}$ is a fundamental basis of $N$.
  Let $A=M[X,V\setminus X]$.
  We have
  $  \rk A=\rk A^t=|X|-\operatorname{nullity}(A^t)$,
  where the \emph{nullity} of $A^t$ is  $\dim(\{ x\in \F^X: A^t x=0\})$,
  that is eqaul to $ \dim(\{x\in \F^X: x^t A=0\} ) $.
  
  Let $\varphi:\F^V\rightarrow N$ be a linear transformation with 
  $\varphi(p)= \sum_{v\in V} p(v)f_v$. 
  Then, $\varphi$ is an isomorphism and therefore we have the
  following:
  \begin{align*}
    \dim(N\times X) 
    &= \dim(\{ y\in N:  y(j)=0\text{ for all }j\in V\setminus X\})\\
    &= \dim(\varphi^{-1}(
    \{ y\in N:  y(j)=0\text{ for all }j\in V\setminus X\}
    ))\\
    &= \dim(\{  x\in \F^V:
   \sum_{i\in V} x(i) f_i(j) = 0 \text{ for all } j\in V\setminus
    X\})\\
    &= \dim (\{x\in \F^X: 
    \sum_{i\in X} x(i) m_{ij} =0 \text{ for all }j\in V\setminus X
    \})\\
    &= \dim( \{x\in \F^X:
    x^t A = 0 \})\\
    &= \operatorname{nullity}(A^t).
  \end{align*}
  We deduce that $\rk A=|X|-\dim(N\times X)$.
\end{proof}
The above theorem gives the following corollaries.
\begin{COR}
  Let $\F$ be a field and let $N$ be a Lagrangian chain-group on $V$ to $K=\F^2$.
  If $M_1$ and $M_2$ are two fundamental matrices of $N$, 
  then 
  $\rk M_1[X,V\setminus X]=\rk M_2[X,V\setminus X]$ for all
  $X\subseteq V$.
\end{COR}
\begin{COR}
  Let $M$ be a skew-symmetric or symmetric $V\times V$ matrix over a
  field $\F$.
  Let $N$ be a Lagrangian chain-group on $V$ to $K=\F^2$ such that
  $(N,a,b)$ is a matrix representation of $N$.
  Then the rank-width of $M$ is equal 
  to the branch-width of $N$.
\end{COR}

\section{Generalization of Tutte's linking theorem}\label{sec:tutte}
We prove an analogue of Tutte's linking theorem \cite{Tutte1965} 
for Lagrangian chain-groups.
Tutte's linking theorem is a generalization of Menger's theorem of
graphs to matroids.
Robertson and Seymour~\cite{RS1990} uses Menger's theorem extensively
for proving well-quasi-ordering of graphs of bounded tree-width. When
generalizing this result to matroids, Geelen, Gerards, and
Whittle~\cite{GGW2002} used Tutte's linking theorem for matroids. 
To further generalize this to
Lagrangian chain-groups, we will need a generalization of Tutte's
linking theorem for Lagrangian chain-groups.

A crucial step for proving this is to ensure that the connectivity
function 
behaves nicely
on one of two minors $N\delete \{v\}$ and $N\contract
\{v\}$ of a Lagrangian chain-group $N$. 
The following inequality was observed by Bixby~\cite{Bixby1982} for matroids.
\begin{PROP}\label{prop:tutte}
  Let $v\in V$. Let $N$ be a chain-group on $V$ to $K=\F^2$ and let
  $X,Y\subseteq V\setminus\{v\}$. Then,
  \[
  \lambda_{N\delete\{v\}}(X)+\lambda_{N\contract\{v\}}(Y)
  \ge \lambda_N(X\cap Y)+\lambda_N(X\cup Y\cup \{v\})-1.
  \]
\end{PROP}

We first prove the following lemma for the above proposition.
\begin{LEM}\label{lem:tutte1}
  Let $v\in V$.
  Let $N$ be a chain-group on $V$ to $K=\F^2$ and 
  let $X,Y\subseteq V\setminus\{v\}$. Then,
  \begin{multline*}
  \dim(N\times (X\cap Y))
  +\dim(N\times (X\cup Y\cup \{v\}))\\
  \ge 
  \dim((N\delete \{v\}) \times X)
  +\dim((N\contract \{v\})\times Y).
\end{multline*}
  Moreover, the equality does not hold if $\dup v\in N$ or $\ddown
  v\in N$.
\end{LEM}
\begin{proof}
  We may assume that $V=X\cup Y\cup \{v\}$.
  Let 
  \begin{align*}
    N_1&=\left\{f\in N: \form{f(v),\tbinom 1 0}_K=0,
    f(x)=0 \text{ for all }x\in V\setminus X\setminus\{v\}\right\},\\
    N_2&=\left\{f\in N: \form{f(v),\tbinom 0 1}_K=0,
    f(x)=0 \text{ for all }x\in V\setminus Y\setminus\{v\}\right\}.
  \end{align*}
  We use the fact that
  $\dim(N_1+N_2)+\dim(N_1\cap N_2)=\dim(N_1)+\dim(N_2)$.
  It is easy to see that if $f\in N_1\cap N_2$, then $f(v)=0$ and
  therefore $(N_1\cap N_2)\cdot (X\cap Y)=N\times (X\cap Y)$ and
  $\dim(N_1\cap N_2)=\dim (N\times (X\cap Y))$.
  Moreover, $N_1+N_2\subseteq N$ and therefore $\dim(N)\ge
  \dim(N_1+N_2)$.
  It is clear that $\dim(N\delete\{v\}\times X)\le \dim N_1$
  and $\dim(N\contract \{v\}\times X)\le \dim N_2$.
  Therefore we conclude that 
  $\dim(N\times (X\cap Y))+\dim N\ge \dim(N\delete \{v\}\times X)
  +\dim(N\contract \{v\}\times Y)$.

  If $\dup v\in N$, then $\dim(N\delete\{v\}\times X)<\dim N_1$ and
  therefore the equality does not hold. Similarly if $\ddown v\in N$,
  then the equality does not hold as well.
\end{proof}
\begin{proof}[Proof of Proposition~\ref{prop:tutte}]
  Since $N$ and $N^\bot$ have the same connectivity function $\lambda$
  and
  $N^\bot\delete\{v\}=(N\delete\{v\})^\bot$, 
  $N^\bot\contract\{v\}=(N\contract\{v\})^\bot$, 
  (Lemma~\ref{thm:delcondual}),
  we may assume that $\dim N-\dim(N\delete \{v\})\in \{0,1\}$ (Proposition~\ref{prop:minordim}) 
  by replacing $N$ by $N^\bot$ if necessary.
  Let $X'=V\setminus X\setminus\{v\}$ and $Y'=V\setminus
  Y\setminus\{v\}$.
  We recall that 
  \begin{align*}
    2\lambda_N & (X\cap Y)
    \\&=\dim N
   -\dim(N\times (X\cap Y))-
    \dim(N\times (X'\cup Y'\cup\{v\})),\\
    2\lambda_N & (X\cup Y\cup \{v\})
    \\
    &=\dim N
   -\dim (N\times (X\cup Y\cup \{v\}))
    -\dim(N\times (X'\cap Y')),\\
    2\lambda_{N\delete\{v\}} &(X)\\
   &= \dim (N\delete\{v\})
   -\dim(N\delete\{v\}\times X)-\dim(N\delete\{v\}\times X'),\\
   2\lambda_{N\contract\{v\}}&(Y)\\
   &= \dim (N\contract\{v\})
   -\dim(N\contract\{v\}\times Y)-\dim(N\contract\{v\}\times Y').
  \end{align*}
  It is easy to deduce this lemma from Lemma~\ref{lem:tutte1} if 
  \begin{equation}
    \label{eq:1}
    2\dim N-
    \dim(N\delete \{v\})-\dim(N\contract \{v\})\le 2.
  \end{equation}
  Therefore we may assume that \eqref{eq:1} is false.
  Since we have assumed that $\dim N-\dim (N\delete\{v\})\in \{0,1\}$, 
  we conclude that $\dim N-\dim(N\contract \{v\})\ge 2$.
  By Proposition~\ref{prop:minordim}, we have $\ddown v\in N$.
  Then the equality in the inequality of Lemma~\ref{lem:tutte1} does
  not hold. So, we conclude that 
  $  \dim(N\times (X\cap Y))
  +\dim(N\times (X\cup Y\cup \{v\}))
  \ge 
  \dim(N\delete \{v\} \times X)
  +\dim(N\contract \{v\} \times Y)+1$
  and the same inequality for $X'$ and $Y'$. Then,
  $\lambda_{N\delete\{v\}}(X)+\lambda_{N\contract\{v\}}(Y)
  \ge \lambda_N(X\cap Y)+\lambda_N(X\cup Y\cup \{v\})
  -3/2+1$.
\end{proof}
We are now ready to prove an analogue
 of Tutte's linking theorem for Lagrangian chain-groups.
\begin{THM}\label{thm:tutte}
  Let $V$ be a finite set and $X$, $Y$ be disjoint subsets of $V$. 
  Let $N$ be a Lagrangian chain-group on $V$ to $K$. 
  The following two conditions are equivalent:
  \begin{enumerate}[(i)]
  \item $\lambda_N(Z)\ge k$ for all sets $Z$ such that $X\subseteq
    Z\subseteq V\setminus Y$,
  \item there is a minor $M$ of $N$ on $X\cup Y$ such that 
    $\lambda_M (X)\ge k$.
  \end{enumerate}
  In other words,
  \begin{multline*}
  \min
  \{\lambda_N(Z):X\subseteq Z\subseteq V\setminus Y\}\\
  =
  \max
  \{\lambda_{N\delete U\contract W} (X)
  :U\cup W=V\setminus (X\cup Y), U\cap W=\emptyset
  \}.
  \end{multline*}
\end{THM}
\begin{proof}
  By Theorem~\ref{thm:connectivityminor}, (ii) implies (i).
  Now let us assume (i) and show (ii). 
  We proceed by induction on $|V\setminus (X\cup Y)|$.
  If $V=X\cup Y$, then it is trivial. So we may assume that
  $|V\setminus (X\cup Y)|\ge 1$.
  Since $\lambda_N(X)$ are integers
  for all $X\subseteq V$ by Lemma~\ref{lem:even}, we may assume
  that
  $k$ is an integer.
  
  Let $v\in V\setminus (X\cup Y)$. 
  Suppose that (ii) is false. 
  Then there is no minor $M$ of
  $N\delete\{v\}$  or $N\contract \{v\}$ 
  on $X\cup Y$ having $\lambda_M(X)\ge k$. 
  By the induction hypothesis, we conclude that there are sets $X_1$ and
  $X_2$ such that
  $X\subseteq X_1\subseteq V\setminus Y\setminus\{v\}$,
  $X\subseteq X_2\subseteq V\setminus Y\setminus\{v\}$,
  $\lambda_{N\delete \{v\}}(X_1)<k$,
  and $\lambda_{N\contract\{v\}}(X_2)<k$.
  By Lemma~\ref{lem:even}, $\lambda_{N\delete\{v\}}(X_1)$ and 
  $\lambda_{N\contract\{v\}}(X_2)$ are integers.
  Therefore   $\lambda_{N\delete \{v\}}(X_1)\le k-1$
  and $\lambda_{N\contract\{v\}}(X_2)\le k-1$. 
  By Proposition~\ref{prop:tutte},
  \[
  \lambda_{N\delete\{v\}}(X_1)+\lambda_{N\contract\{v\}}(X_2)
  \ge \lambda_N(X_1\cap X_2)+\lambda_N(X_1\cup X_2\cup \{v\})-1.
  \]
  This is a contradiction because $\lambda_N(X_1\cap X_2)\ge k$ and
  $\lambda_N(X_1\cup X_2\cup \{v\})\ge k$.
\end{proof}
\begin{COR}\label{cor:tutte}
  Let $N$ be a Lagrangian chain-group on $V$ to $K$ and let
  $X\subseteq Y\subseteq V$.
  If $\lambda_N(Z)\ge \lambda_N(X)$ for all $Z$ satisfying 
  $X\subseteq Z\subseteq Y$, then 
  there exist disjoint subsets $C$ and $D$ of $Y\setminus X$ such that 
  $C\cup D=Y\setminus X$ and 
  $N\times X= N\times Y\contract C\delete D$.
\end{COR}
\begin{proof}
  For all $C$ and $D$ if $C\cup D=Y\setminus X$ and $C\cap
  D=\emptyset$, then 
  $N\times X\subseteq N\times Y\contract C\delete D$.
  So it is enough to show that there exists a partition $(C,D)$ of
  $Y\setminus X$ such that
  \[
  \dim(N\times X)\ge \dim(N\times Y\contract C\delete D).
  \]
  By Theorem~\ref{thm:tutte}, there is a minor $M=N\contract C\delete
  D$   of $N$ on $X\cup
  (V\setminus Y)$ such that $\lambda_M(X)\ge \lambda_N(X)$.
  It follows that 
  $|X|-\dim(N\contract C\delete D\times X)\ge 
  |X|-\dim(N\times X)$.
  Now we use the fact that $N\contract C\delete D\times X=
  N\times Y \contract C\delete D$.
\end{proof}

\section{Well-quasi-ordering of Lagrangian chain-groups}\label{sec:wqochain}
In this section, we prove that Lagrangian chain-groups of bounded
branch-width are well-quasi-ordered under taking a minor. Here we
state its simplified form.
\begin{THM}[Simplified]\label{thm:main}
  Let $\F$ be a finite field
  and let $k$ be a constant.
  Every infinite sequence $N_1,N_2,\ldots$ of Lagrangian
  chain-groups over $\F$
  having branch-width at most $k$
  has a pair $i<j$ such that 
  $N_i$ is simply isomorphic to a minor of $N_j$.
\end{THM}
This simplified version is enough to obtain results in
Sections~\ref{sec:wqomat} and~\ref{sec:matroid}.
One may first read corollaries in later sections and return
to this section.

\subsection{Boundaried chain-groups}
For an isotropic chain-group $N$ on $V$ to $K=\F^2$, we write $N^\bot/N$
for a vector space over $\F$ containing vectors of the form $a+N$ where
$a\in N^\bot$ such that
\begin{enumerate}[(i)]
\item $a+N=b+N$ if and only if $a-b\in N$,
\item $(a+N)+(b+N)=(a+b)+N$,
\item $c(a+N)=ca+N$ for $c\in \F$.
\end{enumerate}
An \emph{ordered basis} of a vector space is a sequence of vectors in the
vector space such that the vectors in the sequence form a basis
of the vector space.
An ordered basis of $N^\bot/N$ is called a \emph{boundary} of $N$.
An isotropic chain-group $N$ on $V$ to $K$ with a boundary $B$ is called
a \emph{boundaried chain-group} on $V$ to $K$, denoted by $(V,N,B)$.

By the theorem in the linear algebra, we know that
\[|B|=\dim(N^\bot)-\dim(N)=2(|V|-\dim N).\]

We define contractions and deletions of boundaries $B$ of an isotropic
chain-group $N$ on $V$ to $K$.
Let $B=\{b_1+N,b_2+N,\ldots,b_m+N\}$ be a boundary of $N$.
For a subset $X$ of $V$, if $|V\setminus X|-\dim(N\delete X)=|V|-\dim
N$, then we define $B\delete X$ as a sequence 
\[\{b_1'\cdot(V\setminus X)+N\delete X,
b_2'\cdot(V\setminus X)+N\delete X,
\ldots,
b_m'\cdot(V\setminus X)+N\delete X\}\]
where
$b_i+N=b_i'+N$
and $\form{b_i'(v),\tbinom 10}_K=0$ for all $v\in X$.
Similarly if 
$|V\setminus X|-\dim(N\contract X)=|V|-\dim
N$, then we define $B\contract X$ as a sequence 
\[\{b_1'\cdot(V\setminus X)+N\contract X,
b_2'\cdot(V\setminus X)+N\contract X,
\ldots,
b_m'\cdot(V\setminus X)+N\contract X\}\]
where
$b_i+N=b_i'+N$
and $\form{b_i'(v),\tbinom 01}_K=0$ for all $v\in X$.
We prove that $B\delete X$ and $B\contract X$ are well-defined.

\begin{LEM}\label{lem:boundaryminor}
  Let $N$ be an isotropic chain-group on $V$ to $K$.
  Let $X$ be a subset of $V$.
  If $\dim N-\dim (N\delete X)=|X|$ and 
  $f\in N^\bot$, then there exists a chain $g\in N^\bot$ such that
  $f-g\in N$ and 
  $\form{g(x),\tbinom 10}_K=0$ for all $x\in X$.
\end{LEM}
\begin{proof}
  We proceed by induction on $|X|$. If $X=\emptyset$, then it is
  trivial.
  Let us assume that $X$ is nonempty.
  Notice that $N\subseteq N^\bot$ because $N$ is isotropic.
  We may assume that there is $v\in X$ such that
  $\form{f(v),\tbinom10}_K\neq 0$, because otherwise we can take
  $g=f$.


  Then $\dup v\notin N$.
  Since $|V\setminus X|-\dim (N\delete X)=|V|-\dim N$,
  we have $|V|-1-\dim (N\delete \{v\})=|V|-\dim N$ (Corollary~\ref{cor:minordim}) and therefore 
  $\dup v\notin N^\bot$ by Proposition~\ref{prop:minordim}.
  
  Thus there exists a chain $h\in N$ 
  such that $\form{h,\dup
  v}=\form{h(v),\tbinom 10}_K\neq 0$. By multiplying a nonzero
  constant to $h$, we may assume  that \[\form{f(v)-h(v),\tbinom
    10}_K=0.\]
  Let $f'=f-h\in N^\bot$. Then $\form{f'(v),\tbinom 10}_K=0$ and therefore 
  $f'\cdot (V\setminus \{v\})\in N^\bot \delete \{v\}=(N\delete\{v\})^\bot$.
  By using the induction hypothesis based on the fact that 
  $\dim(N\delete\{v\})-\dim (N\delete X)=|X|-1$, we deduce that
  there exists a chain $g'\in (N\delete\{v\})^\bot$ such that 
  $f'\cdot (V\setminus\{v\})-g'\in N\delete\{v\}$ and 
  $\form{g'(x),\tbinom 10}_K=0$ for all $x\in X\setminus\{v\}$.
  Let $g$ be a chain in $N^\bot$ such that $g\cdot
  (V\setminus\{v\})=g'$ and $\form{g(v),\tbinom 10}_K=0$.
  
  We know that $\form{f'(v)-g(v),\tbinom 10}_K=0$.
  Since $(f'-g)\cdot (V\setminus \{v\})\in N\delete \{v\}$
  and $\dup v\notin N$, we deduce that $f'-g\in N$.
  Thus $f-g=f'-g+h\in N$.
  Moreover for all $x\in X$, $\form{g(x),\tbinom 10}_K=0$.
\end{proof}
\begin{LEM}\label{lem:boundaryminor2}
  Let $N$ be an isotropic chain-group on $V$ to $K$.
  Let $X$ be a subset of $V$. 
  Let   $f$ be a chain in $N^\bot$ such that
  $\form{f(x),\tbinom 10}_K=0$ if $x\in X$
  and 
  $f(x)=0$ if $x\in V\setminus X$.
  If $\dim N-\dim(N\delete X)=|X|$,
  then 
  $f\in N$.
\end{LEM}
\begin{proof}
  We proceed by induction on $|X|$. 
  We may assume that $X$ is nonempty. Let $v\in X$.
  By Corollary~\ref{cor:minordim}, 
   $\dim(N\delete\{v\})=\dim N-1$
   and $\dim (N\delete\{v\})-\dim (N\delete X)=|X|-1$.
   Proposition~\ref{prop:minordim} implies that either $\dup v \in
   N$ or $\dup v\notin N^\bot$.

  
   By Theorem~\ref{thm:delcondual}, $f\cdot (V\setminus \{v\})\in
   (N\delete\{v\})^\bot$.
   By the induction hypothesis, $f\cdot (V\setminus\{v\})\in N\delete
   \{v\}$.
   There is a chain $f'\in N$ such that $f'(x)=f(x)$ for all $x\in
   V\setminus\{v\}$
   and $\form{f'(v),\tbinom10}_K=0$.
   Then $f-f'=cv^*$ for some $c\in \F$ by Lemma~\ref{lem:basic}.
   Because $N$ is isotropic, $f-f'\in N^\bot$.

   If $\dup v\in N$, then $f=f'+cv^*\in N$.
    If $\dup v\notin N^\bot$, then $c=0$ and therefore $f\in N$.
\end{proof}

  


\begin{PROP}
  Let $N$ be an isotropic chain-group on $V$ to $K$ with a boundary
  $B$.
  Let $X$ be a subset of $V$. 
  If $|V\setminus X|-\dim (N\delete X)=|V|-\dim N$, then 
  $B\delete X$ is well-defined and it is a boundary of $N\delete X$.
  Similarly 
  if $|V\setminus X|-\dim (N\contract X)=|V|-\dim N$, then 
  $B\contract X$ is well-defined and it is a boundary of $N\contract X$.
\end{PROP}
\begin{proof}
  By symmetry it is enough to show for $B\delete X$.
  Let $B=\{b_1+N,b_2+N,\ldots,b_m+N\}$.

  By Lemma~\ref{lem:boundaryminor}, there exists a chain $b_i'\in
  N^\bot$ such that $b_i+N=b_i'+N$ and $\form{b_i'(x),\tbinom 10}_K=0$
  for all $x\in X$. 

  Suppose that there are chains $c_i$ and $d_i$ in $N^\bot$
  such that 
  $b_i+N=c_i+N=d_i+N$ and
  $\form{c_i(x),\tbinom 10}_K=\form{d_i(x),\tbinom 10}_K=0$ for all
  $x\in X$.
  Since $c_i-d_i\in N$ and $\form{c_i(x)-d_i(x),\tbinom 10}_K=0$ for
  all $x\in X$, we deduce that $(c_i-d_i)\cdot (V\setminus X)\in
  N\delete X$ and therefore 
  \[c_i\cdot (V\setminus X)+N\delete X=d_i\cdot (V\setminus X)+N\delete X.\]
  Hence $B\delete X$ is well-defined.

  Now we claim that $B\delete X$ is a boundary of $N\delete X$.
  Since $\dim((N\delete X)^\bot/ (N\delete X))
  = 2|V\setminus X|-2\dim(N\delete X)=2|V|-2\dim N=\dim
  N^\bot/N=|B|=|B\delete X|$, 
  it is enough to show that 
  $B\delete X$ is  linearly independent in $(N\delete
  X)^\bot /N\delete X$.
  We may assume that $\form{b_i(x),\tbinom 10}_K=0$ for all $x\in X$.
  Let $f_i=b_i\cdot (V\setminus X)\in N^\bot\delete X$.
  We claim that $\{f_i+N\delete X: i=1,2,\ldots,m\}$ is linearly independent. 
  Suppose that 
  $  \sum_{i=1}^m a_i (f_i+N\delete X)=0  $
  for some constants $a_i\in \F$.
  This means $\sum_{i=1}^m a_i f_i\in N\delete X$.
  Let $f$ be a chain in $N$ such that
  $f\cdot (V\setminus X)=\sum_{i=1}^m a_i f_i$ and
  $\form{f(x),\tbinom 10}_K=0$ for all $x\in X$.
  Let $b=\sum_{i=1}^m a_i b_i$.
  Clearly $b\in N^\bot$.
  
  We consider the chain $b-f$. 
  Since $N$ is isotropic, $f\in N^\bot$ and so $b-f\in N^\bot$.
  Moreover
  $(b-f)\cdot (V\setminus X)=0$ and $\form{b(x)-f(x),\tbinom
    10}_K=0$ for all $x\in X$.
  By Lemma~\ref{lem:boundaryminor2}, we deduce that 
  $b-f\in N$ and therefore $b=(b-f)+f\in N$.
  Since $B$ is a basis of $N^\bot/N$, $a_i=0$ for all $i$.
  We conclude that $B\delete X$ is linearly independent.
\end{proof}

A boundaried chain-group $(V',N',B')$ is a \emph{minor} of
another boundaried chain-group $(V,N,B)$ if 
\[
|V'|-\dim N'= |V|-\dim N
\]
and 
there exist disjoint subsets $X$ and $Y$ of $V$ such that
$V'=V\setminus(X\cup Y)$,
$N'=N\delete X\contract Y$,
and $B'=B\delete X\contract Y$. 

\begin{PROP}
  A minor of a minor of a boundaried chain-group is a minor of the
  boundaried chain-group.
\end{PROP}
\begin{proof}
  Let $(V_0,N_0,B_0)$, $(V_1,N_1,B_1)$, $(V_2,N_2,B_2)$ be boundaried
  chain-groups.
  Suppose that for $i\in \{0,1\}$, $(V_{i+1},N_{i+1},B_{i+1})$ is a minor of 
  $(V_i,N_i,B_i)$ as follows:
  \[ 
  N_{i+1}=N_i \delete X_i \contract Y_i,
  \quad
  B_{i+1}=B_i \delete X_i \contract Y_i.
  \]
  It is easy to deduce that $|V_0|-\dim N_0=|V_2|-\dim N_2$
  and $N_2=N_0\delete (X_0\cup X_1)\contract (Y_0\cup Y_1)$.

  We claim that $B_2=B_0\delete (X_0\cup X_1)\contract (Y_0\cup Y_1)$.
  By Corollary~\ref{cor:minordim}, we deduce  that 
  $
  |V_0\setminus (X_0\cup X_1)|-\dim N_0\delete (X_0\cup X_1)
  = |V_0|-\dim N_0
  = |V_2|-\dim N_2$
  and so it is possible to delete $X_0\cup X_1$ from $V_0$
  and then contract $Y_0\cup Y_1$.
  From the definition, it is easy to show that 
  $B\delete (X_0\cup X_1)\contract (Y_0\cup Y_1)=B_2$.
\end{proof}

\subsection{Sums of boundaried chain-groups}
Two boundaried chain-groups  over the same field
are \emph{disjoint}  if
their ground sets are disjoint.
In this subsection, we define \emph{sums} of disjoint boundaried chain-groups
and their \emph{connection types}.

A boundaried chain-group $(V,N,B)$ over a field $\F$ is 
a \emph{sum} of
disjoint boundaried chain-groups $(V_1,N_1,B_1)$ and $(V_2,N_2,B_2)$ 
over $\F$
if \[N_1=N\times V_1,~ N_2=N\times V_2, \text{ and }V=V_1\cup V_2.\]
For a chain $f$ on $V_1$ to $K$ and
a chain $g$ on $V_2$ to $K$, 
we denote $f\oplus g$ for a chain on $V_1\cup V_2$ to $K$ such that
$(f\oplus g)\cdot V_1=f$ and $(f\oplus g)\cdot V_2=g$.
The \emph{connection type} of the sum
is a sequence $(C_0,C_1,\ldots,C_{|B|})$ of sets of sequences 
in $\F^{|B_1|}\times \F^{|B_2|}$
such that,
for
$B=\{b_1+N,b_2+N,\ldots,b_{|B|}+N\}$,
$B_1=\{b^1_1+N_1,b^1_2+N_1,\ldots,b_{|B_1|}^1+N_1\}$, and 
$B_2=\{b^2_1+N_2,b^2_2+N_2,\ldots,b^2_{|B_2|}+N_2\}$,
\begin{align*}
  C_0&=
  \left\{ (x,y)\in \F^{|B_1|}\times \F^{|B_2|}:
    \left( \sum_{i=1}^{|B_1|} x_i b_i^1\right)
    \oplus \left(\sum_{j=1}^{|B_2|} y_j b_j^2 \right)\in N\right\},\\
  \intertext{and for $s\in\{1,2,\ldots,|B|\}$, }
 C_s&=
  \left\{  (x,y)\in \F^{|B_1|}\times \F^{|B_2|}:
    \left( \sum_{i=1}^{|B_1|} x_i b_i^1\right)
    \oplus \left(\sum_{j=1}^{|B_2|} y_j b_j^2 \right)-b_s\in N\right\}.
\end{align*}

\begin{PROP}
  The connection type is well-defined.
\end{PROP}
\begin{proof}
  It is enough to show  that the choices of $b_i$, $b_i^1$, and $b_i^2$
  do not affect $C_s$ for $s\in \{0,1,2,\ldots,|B|\}$.
  Suppose that $b_i+N=d_i+N$, $b_i^1+N_1=d_i^1+N_1$, and
  $b_i^2+N_2=d_i^2+N_2$. 
  Then for every $(x,y)\in \F^{|B_1|}\times \F^{|B_2|}$, 
  \[\sum_{i=1}^{|B_1|}  x_i(b_i^1-d_i^1) 
  \oplus \sum_{j=1}^{|B_2|} y_j(b_j^2-d_j^2)\in N\] 
  because $(b_i^1-d_i^1)\oplus 0\in N$ and 
  $0\oplus (b_j^2-d_j^2)\in N$.
  Moreover if $s\neq 0$, then $b_s-d_s\in N$.
  Hence $C_s$ is well-defined.
\end{proof}
\begin{PROP}
  The connection type uniquely determines the sum of two disjoint
  boundaried chain-groups.
\end{PROP}
\begin{proof}
  Suppose that both $(V,N,B)$ and $(V,N',B')$ are sums of 
  disjoint boundaried chain-groups $(V_1,N_1,B_1)$, $(V_2,N_2,B_2)$ 
  over a field $\F$ with the same connection type
  $(C_0,C_1,\ldots,C_{|B|})$.

  We first claim that $N=N'$. 
  By symmetry, 
  it is enough to show that $N\subseteq N'$.
  Let $a\in N$. Since $a\in N^\bot$ and
  $(N\times {V_1})^\bot=N^\bot\cdot V_1$ by Theorem~\ref{thm:algdual},
  we deduce that 
  $a\cdot V_1\in (N\times V_1)^\bot$ and similarly 
  $a\cdot V_2\in  (N\times V_2)^\bot$.
  Therefore there exists $(x,y)\in \F^{|B_1|}\times \F^{|B_2|}$ such that 
  \[
  f=\sum_{i=1}^{|B_1|} x_i b_i^1 - a\cdot V_1\in N_1
  \quad\text{and}\quad
  g=\sum_{j=1}^{|B_2|} y_j b_j^2 -a\cdot V_2\in N_2.
  \]
  Since $f\oplus 0\in N$ and $0\oplus g\in N$, we have $f\oplus g\in
  N$. We deduce that 
  $\sum_{i=1}^{|B_1|} x_i b_i^1 \oplus \sum_{j=1}^{|B_2|} y_jb_j^2 =
  a+(f\oplus g) \in N$.
  Therefore $(x,y)\in C_0$. So, $a+(f\oplus g)\in N'$ as well. 
  Since $f\oplus 0, 0\oplus g\in N'$, we have $a\in N'$.
  We conclude that $N\subseteq N'$.

  Now we show that $B=B'$. Let $b_s+N$ be the $s$-th element of $B$
  where $b_s\in N^\bot$. 
  Let $b_s'+N$ be the $s$-th element of $B'$ with $b_s'\in N^\bot$.
  Since $b_s\cdot V_1\in (N\times V_1)^\bot$
  and $b_s\cdot V_2\in (N\times V_2)^\bot$,
  there is $(x,y)\in \F^{|B_1|}\times \F^{|B_2|}$ such that 
  \[
  f=\sum_{i=1}^{|B_1|} x_i b_i^1 - b_s\cdot {V_1}\in N_1
  \quad\text{and}\quad
  g=\sum_{j=1}^{|B_2|} y_j b_j^2 -b_s\cdot V_2\in N_2.
  \]
  Since $f\oplus 0, 0\oplus g\in N$, we have $f\oplus g\in  N$. Therefore 
  $\sum_{i=1}^{|B_1|} x_ib_i^1 \oplus \sum_{j=1}^{|B_2|} y_j b_j^2 - b_s \in
  N$. This implies that $(x,y)\in C_s$ and therefore
  $\sum_{i=1}^{|B_1|} x_ib_i^1 \oplus \sum_{j=1}^{|B_2|} y_j b_j^2 - b_s' \in
  N'=N$.
  Thus, $b_s+N=b_s'+N$.
\end{proof}

In the next proposition, we prove that 
minors of a sum of disjoint boundaried chain-groups
are sums of minors of the boundaried chain-groups
with the same connection type. 
\begin{PROP}\label{prop:minorsum}
  Suppose that  a boundaried chain-group $(V,N,B)$ is  a sum of
  disjoint boundaried
  chain-groups 
  $(V_1,N_1,B_1)$, $(V_2,N_2,B_2)$
  over a field $\F$.
  Let  $(C_0,C_1,\ldots,C_{|B|})$ be the connection type of
  the sum. 
  If 
  \begin{align*}
  |V_1\setminus (X\cup Y)|-\dim (N_1\delete X\contract Y)
 &= |V_1|-\dim N_1\\
  \intertext{and}
  |V_2\setminus (Z\cup W)|-\dim (N_2\delete Z\contract W)
  &= |V_2|-\dim N_2,
  \end{align*}
  then $(V\setminus(X\cup Y\cup Z\cup W),N\delete (X\cup Z)\contract
  (Y\cup W),B\delete (X\cup Z)\contract
  (Y\cup W))$ is a well-defined minor of
  $(V,N,B)$.
  Moreover it is a sum of 
  $(V_1\setminus (X\cup Y),N_1\delete X\contract Y,B_1\delete
  X\contract Y)$ and $(V_2\setminus(Z\cup W),N_2\delete Z\contract W,
  B_2\delete Z\contract W)$
  with the connection type $(C_0,C_1,\ldots,C_{|B|})$.
\end{PROP}
\begin{proof}
  We proceed by induction on $|X\cup Y\cup Z\cup W|$. 
  If $X\cup Y\cup Z\cup W=\emptyset$, then it is trivial. 

  Suppose that $|X\cup Y\cup Z\cup W|=1$.
  By symmetry, we may assume that $Y=Z=W=\emptyset$. Let $v\in X$.
  Since $|V_1\setminus \{v\}|-\dim(N_1\delete\{v\}) = |V_1|-\dim N_1$, 
  either $\dup v\in N_1$ or $\dup v\notin N_1^\bot$ by Proposition~\ref{prop:minordim}.
  Since $N_1=N\times V_1$, 
  we deduce that either $\dup v\in N$ or 
  $\dup v\notin N^\bot$. Thus, $|V\setminus\{v\}|-\dim (N\delete\{v\})=|V|-\dim N$
  and so $(V\setminus \{v\},N\delete\{v\},B\delete \{v\})$ is a minor
  of $(V,N,B)$.

  To show that $(V\setminus \{v\},N\delete\{v\},B\delete \{v\})$ is a
  sum of $(V_1\setminus\{v\},N_1\delete\{v\},B\delete\{v\})$ and
  $(V_2,N_2,B_2)$, it is enough to show that 
  \begin{align}
    \label{eq:sum1}
    N\times V_1\delete\{v\}&=N\delete \{v\}\times
    (V_1\setminus\{v\}),\\
    \label{eq:sum2}
    N\times V_2&=N\delete\{v\}\times V_2.
  \end{align}
  It is easy to see \eqref{eq:sum1} and
  $N\times V_2\subseteq N\delete\{v\}\times V_2$.
  We claim that $N\delete\{v\}\times V_2\subseteq N\times V_2$.
  Suppose that $f$ is a chain in $N\delete\{v\}\times V_2$. 
  There exists a chain $f'$  in $N$ such that
  $f'\cdot V_2=f$,
  $\form{f'(v),\tbinom10}_K=0$, and 
  $f'(x)=0$ for all $x\in V\setminus (V_2\cup\{v\})=V_1\setminus\{v\}$.

  If $f'(v)\neq 0$, then $f'\cdot V_1=c\dup v$ for a nonzero $c\in
  \F$ by Lemma~\ref{lem:basic}.
  Since $N_1^\bot=N^\bot\cdot V_1$ (Theorem~\ref{thm:algdual}), 
  we deduce $\dup v=c^{-1}f'\cdot V_1\in N_1^\bot$. 
  Therefore $\dup v\in N_1$ and so $\dup v \in N$.
  We may assume
  that $f'(v)=0$
  by adding a multiple of $\dup v$ to $f'$.
  This implies that $f\in N\times V_2$.
  We conclude \eqref{eq:sum2}.

  Let $(C_0',C_1',\ldots,C_{|B|}')$ be the connection type of the 
  sum of $(V_1\setminus\{v\},N_1\delete\{v\},B_1\delete\{v\})$ and
  $(V_2,N_2,B_2)$.
  Let  $B=\{b_1+N,b_2+N,\ldots,b_{|B|}+N\}$,
  $B_1=\{b^1_1+N_1,b^1_2+N_1,\ldots,b_{|B_1|}^1+N_1\}$, and 
  $B_2=\{b^2_1+N_2,b^2_2+N_2,\ldots,b^2_{|B_2|}+N_2\}$.
  We may assume that
  $\form {b_i(v),\tbinom 10}_K=0$ and
  $\form {b_i^1(v),\tbinom 10}_K=0$ by Lemma~\ref{lem:boundaryminor}.
  
  We claim that $C_s=C_s'$ for all $s\in \{0,1,\ldots,|B|\}$.
  Let $g$ be a chain in $N^\bot$ such that
  $g=0$ if $s=0$ 
  or $g=b_s$ otherwise.
  If $(x,y)\in C_s$, then 
  \begin{equation}
    \label{eq:sumproof1}
    \left(\sum_{i=1}^{|B_1|} x_i b_i^1 
    \oplus
    \sum_{j=1}^{|B_2|} y_j b_j^2 \right)-g \in N.
  \end{equation}
  Since $\form {b_i^1(v),\tbinom 10}_K=0$
  and $\form{g(v),\tbinom 10}_K=0$, we conclude that 
  \begin{equation}
    \label{eq:sumproof2}
    \left(\sum_{i=1}^{|B_1|} x_i b_i^1 \cdot (V_1\setminus\{v\})
    \oplus
    \sum_{j=1}^{|B_2|} y_j b_j^2\right)
  -g\cdot (V\setminus\{v\}) \in
    N\delete\{v\},
  \end{equation}
  and therefore $(x,y)\in C_s'$. 
  
  Conversely suppose that $(x,y)\in C_s'$.
  Then \eqref{eq:sumproof2} is true. 
  By Lemma~\ref{lem:boundaryminor2}, we deduce 
  \eqref{eq:sumproof1}. Therefore $(x,y)\in C_s$.

  To complete the inductive proof, we now assume that $|X\cup Y\cup
  Z\cup W|>1$. If $X$ is nonempty, let $v\in X$. Let $X'=X\setminus \{v\}$.
  Then, by Corollary~\ref{cor:minordim} we have 
  $|V_1\setminus \{v\}|-\dim  N_1\delete \{v\}=|V_1|-\dim
  N_1$.
  So $(V_1\setminus\{v\},N\delete\{v\},B\delete\{v\})$ is the
  sum of $(V_1\setminus\{v\},N_1\delete\{v\},B_1\delete\{v\})$
  and $(V_2,N_2,B_2)$
  with the connection type $(C_0,C_1,\ldots,C_{|B|})$.
  We deduce our claim by applying the induction hypothesis 
  to $(V_1\setminus\{v\},N_1\delete\{v\},B_1\delete\{v\})$ 
  and $(V_2,N_2,B_2)$.
  Similarly if one of $Y$ or $Z$ or $W$ is nonempty, we deduce our
  claim.
\end{proof}
\subsection{Linked branch-decompositions}
Suppose  $(T,\mathcal L)$ is a branch-decomposition of a
Lagrangian chain-group $N$ on $V$ to $K=\F^2$.
For two edges $f$ and $g$ of $T$,
let $F$ be the set of elements in $V$ corresponding to the leaves in the
component of $T\setminus f$ not containing $g$
and let $G$ be the set of elements in $V$ corresponding to the leaves
in the component of $T\setminus g$ not containing $f$.
Let $P$ be the unique path from $e$ to $f$ in $T$.
We say that $f$ and $g$ are \emph{linked} if 
the minimum width of the edges on $P$
is equal to $\min_{F\subseteq X\subseteq V\setminus G} \lambda_N(X)$.
We say that a branch-decomposition $(T,\mathcal L)$ is \emph{linked}
if every pair of edges in $T$ is linked.

The following lemma is shown by Geelen, Gerards, and
Whittle~\cite{GGW2002,GGW2002cor}.  We state it in terms of Lagrangian
chain-groups, because 
the connectivity function of chain-groups are symmetric
submodular (Theorem~\ref{thm:submodular}).
\begin{LEM}[{Geelen et al.\ \cite[Theorem (2.1)]{GGW2002,GGW2002cor}}]\label{lem:linking}
  A chain-group of branch-width $n$ has
  a linked branch-decomposition of width $n$.
\end{LEM}
Having a linked branch-decomposition will be very useful for proving
well-quasi-ordering because it allows Tutte's linking
theorem to be used.
It was the first step to prove well-quasi-ordering
of matroids of bounded
branch-width by Geelen et al.~\cite{GGW2002}.
An analogous theorem by Thomas \cite{Thomas1990} was used to
prove well-quasi-ordering of graphs of bounded tree-width in
\cite{RS1990}.

\subsection{Lemma on cubic trees}
We use ``lemma on trees,'' proved
by Robertson and Seymour~\cite{RS1990}. It has been used by Robertson and Seymour to prove that a set of graphs of bounded
tree-width is well-quasi-ordered by the graph minor relation. It has been also used
by Geelen et al.\ \cite{GGW2002} 
to prove that a set of matroids representable over a fixed finite
field and having  bounded branch-width 
is well-quasi-ordered by the matroid minor relation.
We need a special case of ``lemma on trees,'' in which a
given forest is cubic, which was also useful for
branch-decompositions
of matroids in \cite{GGW2002}.

The following definitions are  in \cite{GGW2002}.
A \emph{rooted tree} is a finite directed tree where all but one of the 
vertices have indegree 1. A \emph{rooted forest} is a collection of
countably many vertex disjoint rooted trees. Its vertices with
indegree 0 are called \emph{roots} and those with outdegree 0 are
called \emph{leaves}. Edges leaving a root are \emph{root edges} and
those entering a leaf are \emph{leaf edges}.

An \emph{$n$-edge labeling} of a graph $F$ is a map from the set of
edges of $F$
to the set $\{0,1,\ldots,n\}$. Let $\lambda$ be an $n$-edge
labeling of a rooted forest $F$ and let $e$ and $f$ be edges in $F$.
We say that $e$ \emph{is $\lambda$-linked to} $f$ if $F$ contains a
directed path $P$ starting with $e$ and ending with $f$ such that
$\lambda(g)\ge \lambda(e)=\lambda(f)$ for every edge $g$ on $P$.

A \emph{binary forest} is a rooted orientation of a cubic forest with a
distinction between left and right outgoing edges. More precisely, we
call a triple $(F,l,r)$ a \emph{binary forest} 
if $F$ is  a rooted forest where roots have outdegree 1 and $l$
and $r$ are functions defined on non-leaf edges of $F$, such that
the head of each non-leaf edge $e$ of $F$ has exactly two outgoing
edges, namely $l(e)$ and $r(e)$.

\begin{LEM}[{Geelen et al.\ \cite[(3.2)]{GGW2002}}]\label{lem:cubic}
  Let $(F,l,r)$ be an infinite binary forest with an $n$-edge
  labeling $\lambda$. Moreover, let $\le$ be a quasi-order on the set
  of edges of $F$
  with no infinite strictly descending sequences, such
  that $e\le f$ whenever $f$ is $\lambda$-linked to $e$. If the set of
  leaf
  edges of $F$ is well-quasi-ordered by $\le$ but the set of root edges of
  $F$ is not, then $F$ contains an infinite sequence
  $(e_0,e_1,\ldots)$ of non-leaf edges such that 
  \begin{enumerate}[(i)]
  \item $\{e_0,e_1,\ldots\}$ is an antichain with respect to $\le$,
  \item $l(e_0)\le l(e_1)\le l(e_2)\le \cdots$,
  \item $r(e_0)\le r(e_1) \le r(e_2)\le \cdots$.
  \end{enumerate}
\end{LEM}

\subsection{Main theorem}
We are now ready to prove our main theorem. 
To make it more useful, we label each element of the ground set by a
well-quasi-ordered set $Q$  with an ordering $\preceq$
and enforce the minor relation to follow the ordering $\preceq$. 
More precisely, 
for a chain-group $N$ on $V$ to $K$, a \emph{$Q$-labeling} is a mapping from
$V$ to $Q$.
A \emph{$Q$-labeled chain-group} is a chain-group equipped with a
$Q$-labeling. 
A $Q$-labeled chain-group $N'$ on $V'$ to $K$ with a $Q$-labeling $\mu'$ is a
\emph{$Q$-minor} of a $Q$-labeled chain-group $N$ with a $Q$-labeling
$\mu$
if $N'$ is a minor of $N$ 
and $\mu'(v)\preceq \mu(v)$ for all $v\in V'$.

\begin{THMMAIN}[Labeled version]
  Let $Q$ be a well-quasi-ordered set with an ordering $\preceq$.
  Let $k$ be a constant.
  Let $\F$ be a finite field.
  Let $N_1,N_2,\ldots$ be an infinite sequence of $Q$-labeled Lagrangian
  chain-groups over $\F$
  having branch-width at most $k$.
  Then there exist $i<j$ such that 
  $N_i$ is simply isomorphic to a $Q$-minor of $N_j$.
\end{THMMAIN}
\begin{proof}
  We may assume that all bilinear forms $\form{\,,\,}_K$ for all
  $N_i$'s are the same bilinear form, that is either skew-symmetric or
  symmetric
  by taking a subsequence.
  Let $V_i$ be the ground set of $N_i$.
  Let $\mu_i:V_i\to Q$ be the $Q$-labeling of $N_i$.
  We may assume that $|V_i|>1$ for all $i$.
  By Lemma~\ref{lem:linking}, there is a linked branch-decomposition 
  $(T_i,\mathcal L_i)$ of $N_i$ of width at most $k$ for each $i$.
  Let $T$ be  a forest such that the $i$-th component is $T_i$. 
  To make $T$ a binary forest, 
  for each $T_i$, 
  we create a vertex $r_i$ of degree $1$, called a \emph{root}, 
  create a vertex of degree $3$ by subdividing an edge of $T_i$
  and making it adjacent to $r_i$,
  and direct every edge of $T_i$ so
  that each leaf has a directed path from the root $r_i$.

  We now define a $k$-edge labeling $\lambda$ of $T$, necessary for Lemma~\ref{lem:cubic}.
  For each edge $e$ of $T_i$, let $X_e$ be the set of leaves of $T_i$
  having a directed path from $e$. Let $A_e=\mathcal L_i^{-1}(X_e)$.
  We let $\lambda(e)=\lambda_{N_i}(A_e)$.  

  We want to associate each edge $e$  of $T_i$ 
  with a $ Q$-labeled boundaried chain-group
  $P_e=(A_e,N_i\times A_e,B_e)$ with a $Q$-labeling
  $\mu_e=\mu_i|_{A_e}$
  and some boundary $B_e$ satisfying  the following property:
  \begin{equation}
    \label{eq:2}
    \text{if $f$ is $\lambda$-linked to $e$, then 
      $P_e$ is a $Q$-minor of $P_f$.}
  \end{equation}
  We note that $\mu_i|_{A_e}$ is  a function on $A_e$ such that
  $\mu_i|_{A_e}(x)=\mu_i(x)$ for all $x\in A_e$.

  We claim that we can assign $B_e$ to satisfy \eqref{eq:2}. We prove
  it by induction on the length of the directed path from the root
  edge of $T_i$ to an edge $e$ of $T_i$.
  If no other edge is $\lambda$-linked to $e$, 
  then let $B_e$ be an arbitrary boundary of $N_i\times A_e$.
  If $f$, other than $e$, is $\lambda$-linked to $e$, then choose $f$
  such that the distance between $e$ and $f$ is minimal. We claim that
  we can obtain
  $B_e$ from $B_f$ by Corollary~\ref{cor:tutte} (Tutte's
  linking theorem) as follows; since $T_i$ is a linked
  branch-decomposition, for all $Z$, if $A_e\subseteq Z\subseteq A_f$,
  then $\lambda_{N_i}(Z)\ge \lambda_{N_i}(A_e)$. 
  By Corollary~\ref{cor:tutte}, there exist disjoint subsets $C$ and
  $D$ of $A_f\setminus A_e$ such that
  $N\times A_e=N\times A_f\contract C\delete D$.
  Since $|A_e|-\dim N_i\times A_e = |A_f|-\dim N_i\times A_f$, 
  $B_e=B_f\contract C\delete D$ is well defined.
  This proves the claim.

  For $e,f\in E(T)$, we write $e\le f$ when a $Q$-labeled boundaried 
  chain-group $P_e$ is simply isomorphic to a $Q$-minor of $P_f$.
  Clearly $\le$ has no infinitely strictly descending sequences, 
  because 
  there are finitely many boundaried chain-groups on bounded
  number of elements  up to simple isomorphisms
  and furthermore 
  $Q$ is well-quasi-ordered.
  By construction, if $f$ is $\lambda$-linked  to $e$, then $e\le f$.

  The leaf edges of $T$ are well-quasi-ordered because there are only
  finite many distinct boundaried chain-groups on one element up to simple isomorphisms
  and $\mathcal Q$ is well-quasi-ordered.

  Suppose that the root edges are not well-quasi-ordered by the
  relation $\le$. 
  By Lemma~\ref{lem:cubic},
  $T$ contains an infinite sequence $e_0,e_1,\ldots$ of non-leaf edges
  such that 
  \begin{enumerate}[(i)]
  \item $\{e_0,e_1,\ldots\}$ is an antichain with respect to $\le$,
  \item $l(e_0)\le l(e_1)\le \cdots$,
  \item $r(e_0)\le r(e_1)\le \cdots$.
  \end{enumerate}
  Since $\lambda(e_i)\le k$ for all $i$, we may assume that
  $\lambda(e_i)$ is a constant for all $i$, by taking a subsequence.

  The boundaried chain-group $P_{e_i}$ is the sum of 
  $P_{l(e_i)}$ and $P_{r(e_i)}$.
  The number of possible distinct
  connection types for this sum
  is finite,
  because $\F$ is finite and $k$ is fixed, 
  Therefore, we may assume  that the connection types for all sums for
  all $e_i$
  are same for all $i$, by taking a subsequence.
  
  Since $l(e_0)\le l(e_1)$, there exists a simple isomorphism $s_l$ from 
  $A_{l(e_0)}$ to a subset of $A_{l(e_1)}$.
  Similarly, there exists a simple isomorphism $s_r$ from 
  $A_{r(e_0)}$ to a subset of $A_{r(e_1)}$ in $r(e_0)\le r(e_1)$.
  Let $s$ be a function on $A_{e_0}=A_{l(e_0)}\cup A_{r(e_0)}$ such
  that
  $s(v)=s_l(v)$ if $v\in A_{l(e_0)}$ and $s(v)=s_r(v)$ otherwise.
  By Proposition~\ref{prop:minorsum}, $P_{e_0}$ is 
  simply isomorphic to a $Q$-minor of $P_{e_1}$ 
  with the simple isomorphism $s$.
  Since $l(e_0)\le l(e_1)$ and $r(e_0)\le r(e_1)$, we deduce that 
  $P_{e_0}$ is simply isomorphic to a $Q$-minor of 
  $P_{e_1}$ 
  and therefore $e_0\le e_1$. This contradicts to (i).
  Hence we conclude that the root edges are well-quasi-ordered by
  $\le$.
  So there exist $i<j$ such that $N_i$ is simply isomorphic to a
  $Q$-minor of $N_j$.
\end{proof}


\section{Well-quasi-ordering of skew-symmetric or symmetric matrices}\label{sec:wqomat}
In this section, we will prove the following main theorem 
for skew-symmetric or symmetric matrices
from 
Theorem~\ref{thm:main}.
\begin{THM}\label{thm:mainmatrix}
  Let $\F$ be a finite field and let $k$ be a constant. 
  Every infinite sequence $M_1$, $M_2$, $\ldots$  
  of skew-symmetric or
  symmetric matrices over $\F$
  of rank-width at most $k$
 has a pair $i<j$ such that 
  $M_i$ is isomorphic to a principal submatrix of 
 $(M_j/A)$
  for some nonsingular principal submatrix $A$ of $M_j$.
\end{THM}
To move from the principal pivot operation given by
Theorem~\ref{thm:minor}
to a Schur complement, 
we need a finer control how we obtain a matrix representation
under taking a minor of a Lagrangian chain-group.

\begin{LEM}\label{lem:finerminor}
  Let $M_1$, $M_2$ be skew-symmetric or symmetric matrices over a
  field $\F$.
 For $i=1,2$, let $N_i$ be a Lagrangian chain-group
  with a special matrix representation $(M_i,a_i,b_i)$
  where $a_i(v)=\tbinom10$, $b_i(v)=\tbinom01$ for all $v$.
  If $N_1=N_2\contract X\delete Y$, then 
  $M_1$ is a principal submatrix of 
  the Schur complement $(M_2/A)$ 
  of some nonsingular principal submatrix $A$ in $M_2$.
\end{LEM}
\begin{proof}
  For $i=1,2$, let $V_i$ be the ground set of $N_i$.
  We may assume that $X$ is a minimal set having some $Y$ such that
  $N_1=N_2\contract X \delete Y$. 
  We may assume $X\neq \emptyset$, because otherwise we apply
  Lemma~\ref{lem:minor}.
  Note that the Schur
  complement of a $\emptyset\times\emptyset$ submatrix in $M_2$ is
  $M_2$ itself.

  Suppose that $M_2[X]$ is singular.
  Let $a_X$ be a chain on $V_2$ to $K=\F^2$ such that 
  $a_X(v)=\tbinom10$ if $v\notin X$ and $a_X(v)=\tbinom01$ if $v\in
  X$.
  By Proposition~\ref{prop:eulerian}, $a'$ is not an eulerian chain of $N_2$.
  Therefore there exists a nonzero chain $f\in N_2$ such that
  $\form{f(v),a_X(v)}_K=0$ for all $v\in V_2$.
  Then $f\cdot V_1=0$ because 
  $f\cdot V_1\in N_1$
  and 
  $a_1$ is an eulerian chain of $N_1=N_2\contract X\delete Y$.
  There exists $w\in X$ such that $f(w)\neq 0$
  because $a_2$ is an eulerian chain of $N_2$.
  For every chain $g\in N_2$, if $\form{g(v),\tbinom10}_K=0$ for $v\in
  Y$
  and $\form{g(v),\tbinom01}_K=0$ for $v\in X$, 
  then $g(w)=c_gf(w)$ for some $c_g\in \F$ by Lemma~\ref{lem:basic}
  and therefore $g\cdot V_1=(g-c_gf)\cdot V_1\in N_2\contract (X\setminus \{w\})
  \delete (Y\cup \{w\})$.
  This implies that $N_2\contract X\delete Y\subseteq N_2\contract
  (X\setminus \{w\})\delete (Y\cup\{w\})$. 
  Since $\dim (N_2\contract X\delete Y)=\dim( N_2\contract
  (X\setminus \{w\})\delete (Y\cup\{w\}))=|V_1|$, we have
 $N_2\contract X\delete Y= N_2\contract
  (X\setminus \{w\})\delete (Y\cup\{w\})$, contradictory to the
  assumption that $X$ is minimal.
  This proves that $M_2[X]$ is nonsingular.

  By Proposition~\ref{prop:chainpivot}, 
  $(M',a',b')$ is another special matrix representation of $N_1$
  where $M'=M*X$ if $\form{\,,\,}_K$ is symmetric
  or $M'=I_X (M*X)$ if $\form{\,,\,}_K$ is skew-symmetric
  and $a'$, $b'$ are given in Proposition~\ref{prop:chainpivot}.
  We observe that $a'\cdot V_1=a_1$ and $b'\cdot V_1=b_1$.
  We apply Lemma~\ref{lem:minor} to deduce that $(M'[V_1],a_1,b_1)$ is
  a matrix representation of $N_1$. This implies that $M'[V_1]=M_1$.
  Let $A=M_2[X]$.
  Notice that $M'[V_1]=(M_2/A)[V_1]$.
 This proves the lemma.
\end{proof}

\begin{proof}[Proof of Theorem~\ref{thm:mainmatrix}]
  By taking an infinite subsequence,  we may assume that all of the
  matrices in the sequence are skew-symmetric or symmetric.
  Let $K=\F^2$ and assume $\form{\,,\,}_K$ is a bilinear form that is
  symmetric if the matrices are skew-symmetric
  and skew-symmetric if the matrices are symmetric.
  Let $N_i$ be the Lagrangian chain-group represented by a matrix
  representation $(M_i,a_i,b_i)$ where $a_i(x)=\tbinom10$,
  $b_i(x)=\tbinom01$ for all $x$.
  Then by Theorem~\ref{thm:main}, there are $i<j$ such that $N_i$ is
  simply isomorphic to a minor of $N_j$. 
  By Lemma~\ref{lem:finerminor}, we deduce the conclusion.
\end{proof}

Now let us consider the notion of delta-matroids, a generalization of matroids.
Delta-matroids lack the notion of the connectivity
and hence it is not clear how to define the branch-width naturally
for delta-matroids. 
We define the branch-width of a $\F$-representable
delta-matroid
as the minimum rank-width of all skew-symmetric or symmetric matrices
over $\F$
representing the delta-matroid.
Then we can deduce the following theorem from Theorem~\ref{thm:repdelta} and Proposition~\ref{prop:matdelta}.
\begin{THM}\label{thm:maindelta}
  Let $\F$ be a finite field and $k$ be a constant. 
  Every infinite sequence 
  $\MM_1$, $\MM_2$, $\ldots$ 
  of $\F$-representable delta-matroids of branch-width at most $k$
  has a pair $i<j$ such that 
  $\MM_i$ is isomorphic to a minor of $\MM_j$.
\end{THM}

\begin{proof}
  Let $M_1$, $M_2$, $\ldots$ be an infinite sequence of skew-symmetric
  or symmetric matrices over $\F$ such that the rank-width of $M_i$ is equal
  to the branch-width of $\MM_i$ and $\MM_i=\MM(M_i)\Delta X_i$.
  We may assume that $X_i=\emptyset$ for all $i$.
  By Theorem~\ref{thm:mainmatrix}, there are $i<j$ such that 
  $M_i$ is isomorphic to a principal submatrix of 
  the Schur complement of a nonsingular principal submatrix in $M_j$.
  This implies that $\MM_i$ is a minor of $\MM_j$ as a delta-matroid.
\end{proof}

In particular, when $\F=GF(2)$, then binary skew-symmetric matrices 
correspond to adjacency matrices of simple graphs.
Then taking a pivot on such matrices is equivalent to taking a
sequence of graph pivots on the corresponding graphs.
We say that a simple graph $H$ is a \emph{pivot-minor} of a simple graph $G$
if $H$ is obtained from $G$ by applying pivots and deleting vertices.
As a matter of a fact, a pivot-minor of a simple graph corresponds to a minor of an even binary
delta-matroid.
The \emph{rank-width} of a simple graph is defined to be the rank-width of
its adjacency matrix over $\F$.
Then Theorem~\ref{thm:mainmatrix} or~\ref{thm:maindelta} implies the
following corollary, originally proved  by Oum~\cite{Oum2004a}.
\begin{COR}[Oum~\cite{Oum2004a}]
  Let $k$ be a constant. Every infinite sequence $G_1$, $G_2$,
  $\ldots$ of simple graphs of rank-width at most $k$ 
   has  a pair $i<j$ such that $G_i$ is isomorphic
  to a pivot-minor of $G_j$.
\end{COR}

\section{Corollaries to matroids and graphs}\label{sec:matroid}
In this section, we will show how Theorem~\ref{thm:main} implies 
the
theorem by Geelen et al.~\cite{GGW2002} on 
well-quasi-ordering 
of $\F$-representable matroids of bounded branch-width for a finite
field $\F$
as well as 
the theorem by Robertson and Seymour~\cite{RS1990} on 
well-quasi-ordering of  graphs of bounded tree-width.

We will briefly review the notion of matroids in the first subsection.
In the second subsection, we will discuss how Tutte chain-groups are
related to representable matroids and Lagrangian chain-groups.
In the last subsection, we deduce the theorem of Geelen et
al.~\cite{GGW2002} 
on matroids 
which in turn implies
the theorem of Robertson and Seymour~\cite{RS1990} on graphs.

\subsection{Matroids}
Let us review matroid theory
briefly. For more on matroid theory, we refer
readers to the book by Oxley \cite{Oxley1992}.

A matroid $M=(E,r)$ is  a
pair formed by a finite set  $E$ of \emph{elements} 
and a \emph{rank} function $r:2^E\rightarrow \mathbb{Z}$ 
satisfying the following axioms:
\begin{enumerate}[i)]
\item $0\le r(X)\le |X|$ for all $X\subseteq E$.
\item If $X\subseteq Y\subseteq E$, then $r(X)\le r(Y)$.
\item  For all $X,Y\subseteq E$,
 $r(X)+r(Y)\ge r(X\cap  Y)+r(X\cup Y)$.
\end{enumerate}
A subset $X$ of $E$ is called \emph{independent} if 
$r(X)=|X|$. 
A \emph{base} is a maximally independent set.
We write $E(M)=E$.
For simplicity, we write $r(M)$ for $r(E(M))$.
For $Y\subseteq  E(M)$, 
$M\setminus Y$ is the  matroid
$(E(M)\setminus Y,r')$ where $r'(X)=r(X)$. 
For $Y\subseteq E(M)$, 
$M/Y$ is the  matroid
$(E(M)\setminus Y,r')$ where $r'(X)=r(X\cup Y)-r(Y)$. 
If $Y=\{e\}$, we denote $M\setminus e=M\setminus
\{e\}$ and $M/e=M/\{e\}$.
It is routine to
prove that $M\setminus Y$ and $M/Y$ are matroids.
Matroids of the form $M\setminus X/Y$ are called a
\emph{minor} of the matroid $M$.

Given a field $\F$ and a set of  vectors in $\F^m$, we can construct a
matroid by letting $r(X)$ be the dimension of the vector space spanned
by vectors in $X$. If a matroid permits this construction, then we
say that the matroid is \emph{$\F$-representable}
or \emph{representable} over $\F$.

The \emph{connectivity function} of a matroid $M=(E,r)$ is 
$\lambda_{M}(X)=r(X)+r(E\setminus X)-r(E)+1$.
A \emph{branch-decomposition} of a matroid $M=(E,r)$ 
is a pair $(T,\mathcal L)$ of a subcubic tree $T$ and
a bijection $\mathcal L:E\rightarrow \{t: \text{$t$ is a leaf of
  $T$}\}$.
For each edge $e=uv$ of the tree $T$, the connected components of
$T\setminus e$ induce a partition $(X_e,Y_e)$ of the leaves of $T$
and we call $\lambda_{M} (\mathcal L^{-1}(X_e))$ the
\emph{width} of $e$.
The \emph{width} of a branch-decomposition $(T,\mathcal L)$ is the
maximum width of all edges of $T$.
The \emph{branch-width} $\bw(M)$ of a matroid $M=(E,r)$
is the minimum width of all its branch-decompositions.
(If $|E|\le 1$, then we define that $\bw(M)=1$.)

\subsection{Tutte chain-groups}
We review Tutte chain-groups \cite{Tutte1971}.
For a finite set $V$ and a field $\F$, 
a \emph{chain} on $V$ to $\F$ is a mapping $f:V\rightarrow \F$.
The \emph{sum} $f+g$ of two chains $f$, $g$ is the chain on $V$ satisfying 
\[(f+g)(x)=f(x)+g(x) \quad\text{for all }x\in V.\]
If $f$ is a chain on $V$ to $\F$ and $\lambda \in \F$, the \emph{product}
$\lambda f$ is a chain on $V$ such that 
\[(\lambda f)(x)=\lambda f(x)\quad\text{for all }x\in V.\]
It is easy to see that the set of all chains on $V$ to $\F$, denoted by
$\F^V$, is a vector space. 
A \emph{Tutte chain-group} on $V$ to $\F$ is a subspace of $\F^V$.
The \emph{support} of  a chain $f$ on $V$ to $\F$ is $\{x\in V: f(x)\neq 0\}$.

\begin{THM}[Tutte \cite{Tutte1965a}]\label{thm:tuttechain}
  Let $N$ be a Tutte chain-group on a finite set $V$ to a field $\F$.
  The minimal nonempty supports of $N$ form the circuits of a
  $\F$-representable
  matroid
  $M\{N\}$ on $V$,
  whose rank is equal to $|V|-\dim N$.
  Moreover every $\F$-representable matroid $M$ admits  a Tutte chain-group $N$
  such that $M=M\{N\}$.
\end{THM}
Let $S$ be a subset of $V$.
For a  chain $f$ on $V$ to $\F$, 
we denote $f\cdot S$ for a chain on $S$ to $\F$ 
such that $(f\cdot S)(v)=f(v)$ for all $v\in S$.
For a Tutte chain-group $N$ on $V$ to $\F$,
we let $N\cdot S=\{f\cdot S: f\in N\}$,
$N\times S=\{f\cdot S: f\in N, f(v)=0\text{ for all }v\notin S\}$,
and $N^\bot=\{g:  \text{$g$ is a chain on $V$ to $\F$},
\sum_{v\in V} f(v)g(v)=0 \text{ for all } f\in N\}$.

A \emph{minor} of a Tutte chain-group $N$ on $V$ to $\F$ is
a Tutte chain-group  of the form $(N\times S)\cdot T$ where
$T\subseteq S\subseteq V$.
By definition, it is easy to see that
$M\{N\}\setminus X=M\{N\times (V\setminus X)\}$
and 
$M\{N\}/ X=M\{N\cdot (V\setminus X)\}$.
So the notion of representable matroid minors is equivalent to 
the notion of Tutte chain-group minors.

Tutte {\cite[Theorem VIII.7.]{Tutte1984}} showed the following theorem. The
proof is basically equivalent to the proof of Theorem~\ref{thm:algdual}.
\begin{LEM}[Tutte {\cite[Theorem VIII.7.]{Tutte1984}}]
  \label{lem:tuttedual}
  If $N$ is a Tutte chain-group on $V$ to $\F$
  and $X\subseteq V$, then 
  $(N\cdot X)^\bot=N^\bot\times X$.
\end{LEM}

We now relate Tutte chain-groups to Lagrangian chain-groups.
For a chain $f$ on $V$ to $\F$,
let $\dup{f}$, $\ddown{f}$ be chains on $V$ to $K=\F^2$
such that
$\dup{f}(v)= \tbinom{f(v)}{0}\in K$,
$\ddown{f}(v)=\tbinom{0}{f(v)}\in K$
for every $v\in V$.
For a Tutte chain-group $N$ on $V$ to $\F$,
we let  $\widetilde N$ be a Tutte chain-group on $V$ to $K$ such that
$\widetilde N= \{\dup{f}+\ddown{g}: f\in N, g\in N^\bot\}$.
Assume that $\form{\,,\,}_K$ is symmetric.

\begin{LEM}\label{lem:chaindim}
  If $N$ is a Tutte chain-group on $V$ to $\F$,
  then $\widetilde N$ is a Lagrangian chain-group on $V$ to $K=\F^2$.
\end{LEM}
\begin{proof}
  By definition, 
  for all $f\in N$ and $g\in N^\bot$, 
  $\form{\dup f,\dup f}=\form{\ddown g,\ddown g}=0$ and 
  $\form{\dup f,\ddown g}=\sum_{v\in V}f(v)g(v)=0$. 
 Thus, $\widetilde N$ is isotropic.
  Moreover, $\dim N+\dim N^\bot= \dim \F^V= |V|$ and therefore
  $\dim\widetilde N=|V|$. (Note that $\widetilde N$ is isomorphic to 
  $N\oplus N^\bot$ as  a vector space.)
 So $\widetilde N$ is a  Lagrangian chain-group.
\end{proof}
\begin{LEM}\label{lem:chainminor}
  Let $N_1$, $N_2$ be Tutte chain-groups on $V_1,V_2$ (respectively) to $\F$.
  Then 
  $N_1$ is a minor of $N_2$ as a Tutte chain-group
  if and only if 
  $\widetilde N_1$ is a minor of $\widetilde N_2$ as a Lagrangian chain-group.
\end{LEM}
\begin{proof}
  Let $N$ be a Tutte chain-group on $V$ to $\F$
  and let $S$ be a subset of $V$.
  It is enough to show that 
  $\widetilde{N\cdot S}= \widetilde N \contract (V\setminus S)$
  and 
  $\widetilde{N\times S}= \widetilde N \delete (V\setminus S)$.

  Let us first show that $\widetilde{N\cdot S}=
  \widetilde N\contract (V\setminus S)$.
  Since $\dim \widetilde{N\cdot S}=
  \dim \widetilde N\contract (V\setminus S)=|S|$ by Lemma~\ref{lem:chaindim}, 
  it is enough to show that $\widetilde{N\cdot S}\subseteq 
  \widetilde N\contract (V\setminus S)$.
  Suppose that $f\in N\cdot S$ and $g\in (N\cdot S)^\bot$. 
  By Lemma~\ref{lem:tuttedual}, $(N\cdot S)^\bot= N^\bot \times S$.
  So there are $\bar f,\bar g\in N$ such that 
  $\bar f\cdot S=f$, $\bar g\cdot S=g$, and $\bar g(v)=0$ 
  for all $v\in  V\setminus S$.
  Now it is clear that 
  $\dup f+\ddown g=
  (\dup {\bar f} +\ddown {\bar g})\cdot S \in N\contract  (V\setminus S)$.

  Now it remains to show that 
  $\widetilde{N\times S}= \widetilde N \delete (V\setminus S)$.
  Let $f\in N\times S$, $g\in (N\times S)^\bot= N^\bot \cdot S$.
  A similar argument shows that 
  $\dup f+\ddown g\in \widetilde N\delete S$
  and therefore 
  $\widetilde{N\times S}\subseteq\widetilde N \delete (V\setminus S)$.
  This proves our claim because these two Lagrangian chain-groups have the
  same dimension.
\end{proof}

Now let us show that
for a Tutte chain-group $N$ on $V$ to $\F$, 
the branch-width of a matroid $M\{N\}$
is 
exactly one more than 
the branch-width of the Lagrangian chain-group $\widetilde N$. 
It is enough to show the following lemma.
\begin{LEM}\label{lem:matroidconn}
  Let $N$ be a Tutte chain-group on $V$ to $\F$.
  Let $X$ be a subset of $V$.
  Then, 
  \[\lambda_{M\{N\}} (X) = \lambda_{\widetilde N} (X) +1.\]
\end{LEM}
\begin{proof}
  Recall that the connectivity function of a matroid is 
  $\lambda_{M\{N\}}(X)=r(X)+r(V\setminus X)-r(V)+1$
  and the connectivity function of a Lagrangian chain-group
  is $\lambda_{\widetilde N}(X)=|X|-\dim (\widetilde N\times X)$. 
  Let $Y=V\setminus X$.
  Let $r$ be the rank function of the matroid $M\{N\}$.
  Then $r(X)$ is equal to the rank of the matroid $M\{N\}\setminus Y
  =M\{N\times X\}$.
  So by Theorem~\ref{thm:tuttechain}, $r(X)=|X|-\dim (N\times X)$.
  Therefore 
  \[\lambda_{M\{N\}}(X)= \dim N-\dim (N\times X)   -\dim (N\times Y)
  +1.\]
  From our construction, 
  $\lambda_{\widetilde N}(X)= |X|-\dim (\widetilde N\times X)
  = |X|-(\dim (N\times X)+\dim (N^\bot\times X))
  =|X|-\dim N\times X -  \dim (N\cdot X)^\bot 
  =|X|-\dim N\times X -  (|X|-\dim N\cdot X)
  =\dim N\cdot X -\dim N\times X$.
  It is enough to show that $\dim N=\dim N\times Y+\dim N\cdot X$.
  Let $L:N\rightarrow N\cdot X$ be a surjective 
  linear transformation such that
  $L(f)=f\cdot X$.
  Then $\dim \ker L= \dim (\{f\in N: f\cdot X=0\})= \dim (N\times Y)$.
  Thus, $\dim N\cdot X=\dim N-\dim N\times Y$.
\end{proof}


\subsection{Application to matroids}

We are now ready to deduce the following theorem by Geelen, Gerards, and
Whittle~\cite{GGW2002} from Theorem~\ref{thm:main}.
\begin{THM}[Geelen, Gerards, and Whittle~\cite{GGW2002}]\label{thm:wqomat}
  Let $k$ be a constant and let $\F$ be a finite field.
  If $M_1,M_2,\ldots$ is an infinite sequence of
  $\F$-representable matroids having branch-width at most $k$, then 
  there exist $i$ and $j$ with $i<j$ such that 
  $M_i$ is isomorphic to a minor of $M_j$.
\end{THM}

To deduce this theorem, we use Tutte chain-groups. 
\begin{proof}
  Let $N_i$ be the Tutte chain-group on $E(M_i)$ to $\F$ such that 
  $M\{N_i\}=M_i$.
  By Lemma~\ref{lem:matroidconn}, the branch-width of the Lagrangian 
  chain-group $\widetilde N_i$ is at
  most $k-1$.
  By Theorem~\ref{thm:main}, there are $i<j$ such that 
  $\widetilde N_i$ is simply isomorphic to a minor of $\widetilde N_j$.
  This implies that $M_i=M\{N_i\}$ is isomorphic to a minor  of
  $M_j=M\{N_j\}$ by Lemma~\ref{lem:chainminor}.
\end{proof}

Geelen et al.~\cite{GGW2002}
showed that Theorem~\ref{thm:wqomat} implies the
following theorem. (We omit the definition of tree-width.)
Thus our theorem also implies the following theorem of Robertson and Seymour.
\begin{THM}[Robertson and Seymour \cite{RS1990}]
  Let $k$ be a constant.
  Every infinite sequence  $G_1,G_2,\ldots$ of 
  graphs having tree-width at most $k$
  has a pair $i<j$ such that 
  $G_i$ is isomorphic to a minor of $G_j$.
\end{THM}


\providecommand{\bysame}{\leavevmode\hbox to3em{\hrulefill}\thinspace}
\providecommand{\MR}{\relax\ifhmode\unskip\space\fi MR }
\providecommand{\MRhref}[2]{%
  \href{http://www.ams.org/mathscinet-getitem?mr=#1}{#2}
}
\providecommand{\href}[2]{#2}


\end{document}